\documentclass[11pt,reqno]{amsart}

\usepackage[a4paper,margin=31mm]{geometry}
\usepackage[foot]{amsaddr}
\usepackage{amsmath,amssymb,amsthm,mathtools}
\usepackage{microtype}
\usepackage{enumitem}

\usepackage[
  backend=biber,
  style=alphabetic,
  sorting=nyt,
  sortcites=true,
  maxbibnames=99,
  maxsortnames=99,
  maxalphanames=4,
  minalphanames=3,
  isbn=false
]{biblatex}

\usepackage{tikz}

\AtBeginBibliography{\raggedright\interlinepenalty=10000}
\usepackage{xcolor}
\usetikzlibrary{calc}
\definecolor{figblue}{HTML}{0072B2}
\definecolor{figorange}{HTML}{D55E00}
\definecolor{figgray}{HTML}{686868}
\tikzset{
  fig vertex/.style={circle,fill=black,draw=none,inner sep=0pt,
                    minimum size=3.6pt},
  fig env/.style={circle,draw=black,fill=white,line width=0.6pt,
                 inner sep=0pt,minimum size=6.4pt},
  fig label/.style={circle,draw=black,fill=white,line width=0.5pt,
                   inner sep=0pt,minimum size=13.5pt,font=\footnotesize},
  fig edge/.style={line width=0.55pt},
  fig dust/.style={draw=figgray,dash pattern=on 0.9pt off 1.4pt,
                  line width=0.85pt},
  fig mass/.style={font=\scriptsize,inner sep=1.3pt},
  fig note/.style={font=\scriptsize,text=black!70,inner sep=1pt},
  fig panel/.style={font=\small,inner sep=0pt},
  fig dustcloud/.pic={
    \fill[figgray] (0,0) circle (0.7pt)
      (0.09,0.015) circle (0.5pt) (-0.09,0.01) circle (0.55pt)
      (0.04,-0.07) circle (0.5pt) (-0.05,-0.07) circle (0.45pt)
      (0,0.08) circle (0.45pt) (0.135,-0.05) circle (0.35pt)
      (-0.13,-0.04) circle (0.35pt) (0.065,0.085) circle (0.3pt)
      (-0.075,0.075) circle (0.3pt);
  },
}
\usepackage[hypertexnames=false]{hyperref}
\usepackage{bookmark}
\usepackage{aliascnt}
\usepackage{zref-clever}
\zcsetup{
  cap=true,
  abbrev=false,
  nameinlink=true,
  counterresetby={
    proposition=section,
    lemma=section,
    corollary=section,
    definition=section,
    remark=section,
    example=section
  }
}
\numberwithin{equation}{section}
\allowdisplaybreaks
\usepackage{mathrsfs}

\hypersetup{
  colorlinks=true,
  linkcolor=blue!55!black,
  citecolor=green!45!black,
  urlcolor=blue!65!black,
  pdftitle={Recursive Paintboxes and the Martin Boundary of the Hoffman Rooted-Tree Graph},
  pdfauthor={Shengjun Zhang},
  pdfsubject={Doob--Martin boundary, recursive paintboxes, central measures, and Poisson--Dirichlet mixtures},
  pdfkeywords={Doob--Martin boundary, recursive paintbox, central measure,
  random recursive tree, Bratteli diagram,
  Poisson--Dirichlet distribution}
}

\theoremstyle{plain}
\newtheorem{theorem}{Theorem}[section]
\newaliascnt{proposition}{theorem}
\newtheorem{proposition}[proposition]{Proposition}
\aliascntresetthe{proposition}
\newaliascnt{lemma}{theorem}
\newtheorem{lemma}[lemma]{Lemma}
\aliascntresetthe{lemma}
\newaliascnt{corollary}{theorem}
\newtheorem{corollary}[corollary]{Corollary}
\aliascntresetthe{corollary}
\theoremstyle{definition}
\newaliascnt{definition}{theorem}
\newtheorem{definition}[definition]{Definition}
\aliascntresetthe{definition}
\newaliascnt{example}{theorem}
\newtheorem{example}[example]{Example}
\aliascntresetthe{example}
\newaliascnt{remark}{theorem}
\newtheorem{remark}[remark]{Remark}
\aliascntresetthe{remark}

\newcommand{\T}{\mathcal T}
\newcommand{\Aut}{\operatorname{Aut}}
\newcommand{\Leaves}{\operatorname{Leaves}}
\newcommand{\rootv}{\operatorname{root}}
\newcommand{\dimf}{\operatorname{dim}}
\newcommand{\Prob}{\mathbb P}
\newcommand{\E}{\mathbb E}
\newcommand{\N}{\mathbb N}
\newcommand{\Nzero}{\mathbb N_0}
\newcommand{\R}{\mathbb R}

\newcommand{\one}{{\mathord{\bullet}}}
\newcommand{\ind}{\mathbf 1}
\newcommand{\rising}[2]{#1^{\overline{#2}}}
\newcommand{\Tcen}{\mathcal T_{\mathrm{cen}}}
\newcommand{\PD}{\operatorname{PD}}
\newcommand{\Ulam}{\mathcal U}
\newcommand{\QSL}{\operatorname{QSL}}
\newcommand{\SL}{\operatorname{SL}}

\title[The Martin boundary of the Hoffman rooted-tree graph]
{Recursive Paintboxes and the Martin Boundary of the Hoffman Rooted-Tree Graph}
\address{Université Paris-Saclay, Faculté des Sciences d’Orsay, Institut de mathématiques d’Orsay, Bâtiment 307, F-91405 Orsay, France}
\email{zhang.shengjun@universite-paris-saclay.fr}
\author{Shengjun Zhang}

\begin{document}
\raggedbottom
\setcounter{tocdepth}{1}

\begin{abstract}
We determine the Martin boundary of the Hoffman graph on unlabelled
non-plane rooted trees and classify the central measures on its path
space. Every boundary point is minimal, and the boundary is parametrized
by deterministic recursive paintboxes up to equality of their finite
sampling laws. The probabilities of sampling trees whose root has a
single child determine the root branch masses and their associated
boundary points, and characterize convergence on the boundary.
Under any central measure, the successive tree shapes converge almost
surely to a random boundary point whose law is the unique mixing measure
in the extremal decomposition. For recursive Ewens trees, this limit is
represented by a random paintbox with independent Poisson--Dirichlet
splits at its vertices, and the mixing measure has no atoms. This applies
in particular to the uniform recursive-tree law and Fulman's rooted-tree
Plancherel law.
\end{abstract}

\maketitle
\tableofcontents

\section{Introduction}\label{sec:introduction}

\subsection{The boundary problem}
We study the Martin boundary of the Hoffman graph on finite unlabelled
non-plane rooted trees \cite{hoffman2003combinatorics}.
For two tree shapes \(s\) and \(t\), the number of edges from \(s\) to \(t\)
is the number of vertices of \(s\) at which adding a leaf produces a tree
of shape \(t\). We write \(s\nearrow t\) when there is such an edge
(see \zcref{fig:hoffman-graph}).
A path records the edge chosen at each step, so distinct paths may have
the same sequence of tree shapes.

A probability measure on infinite paths starting from the single-vertex
tree is called central if, conditional on the tree reached at any finite
level, all paths leading to that tree are equally likely. For a finite
tree \(t\), choose a path to \(t\) uniformly at random. For each fixed
finite path \(\gamma\), consider the probability that the chosen path
begins with \(\gamma\). The boundary problem is to describe all possible
simultaneous limits of these probabilities as \(|t|\to\infty\).

Write \(\T_n\) for the set of tree shapes with \(n\) vertices, set
\(\T=\bigsqcup_{n\ge1}\T_n\), and denote the size of a tree \(t\) by
\(|t|\). Let \(u(t)\) count the paths from the one-vertex tree
to \(t\), and let \(n(s;t)\) count the paths from \(s\) to \(t\). 
We normalize the Martin kernel by setting
\[
    K_s(t)=\frac{u(s)n(s;t)}{u(t)},\qquad |s|\le |t|.
\]
For each path \(\gamma\) to \(s\), the probability that a uniformly chosen
path to \(t\) begins with \(\gamma\) is \(n(s;t)/u(t)\). Summing over the \(u(s)\) paths to \(s\) shows that \(K_s(t)\) is the probability that the path passes through \(s\).
We identify paths to \(t\) with increasing labellings of \(t\) up to
rooted automorphisms. Under this identification, a path passes through
\(s\) precisely when the vertices carrying the \(|s|\) smallest labels
form a subtree of shape \(s\).

For the random recursive tree process on Harris trees, the Martin
boundary is known from the trickle-down construction
\cite[Section~6.3]{evans2012trickle}. Harris trees retain sibling birth
order even after arrival labels are forgotten
\cite[Section~2.1]{grubel2015recursive}, so the first non-root label must
lie in the first root branch. In a uniformly chosen increasing labelling
of a fixed representative of an unordered tree \(t\), a root branch \(a\)
receives this label with probability \(|a|/(|t|-1)\).
\zcref[S]{ex:ordered-unordered-history,ex:ordered-boundary-nonextension}
illustrate this difference and show that the map that forgets sibling order
does not extend continuously from the Martin compactification of Harris trees to that of the Hoffman graph.

In \cite{zhang2026kingman}, Kingman paintboxes parametrize both the
normalized nonnegative characters of the Grossman--Larson algebra and
the boundary of the subgraph consisting of trees whose root branches are
chains. Mixtures of the corresponding extremal central measures form a
proper exposed face of the simplex studied here.

We characterize the limiting sampling laws for the Hoffman graph using recursive paintboxes. At the root, a paintbox
assigns masses to countably many boxes, with the remaining mass treated as
dust, and associates another recursive paintbox with each box of positive mass. Given a finite ordered set of labels, place the least label at the root and let each remaining label independently choose a box or dust according to these masses. 
Each dust label becomes a leaf attached to the root. Within
each occupied box, apply the same procedure independently to its label set
and paintbox, and attach the resulting subtree to the root
(see \zcref{fig:paintbox-sample}). Two paintboxes are sampling equivalent
if they give the same shape law at every sample size. We prove that these equivalence
classes form the full Martin boundary and that every boundary point is
minimal.

\subsection{Main results}
Let \(\Omega_M\) be the Martin compactification of \(\T\), with full
boundary \(\partial\Omega_M\) and minimal boundary \(\partial_{\min}\Omega_M\).
Let \(\partial_{\mathrm{RP}}\) be the set of deterministic recursive
paintboxes modulo sampling equivalence, with the topology of convergence of
all finite sampling laws. For \(\xi\in\partial_{\mathrm{RP}}\), denote
these laws by \(M_n^\xi\).

\noindent\textbf{Theorem A (Boundary classification).}
The sampling map is a homeomorphism of compact metrizable spaces
\begin{equation}\label{eq:intro-main-boundary}
    \Psi:\partial_{\mathrm{RP}}
       \xrightarrow{\ \cong\ }
       \partial\Omega_M=\partial_{\min}\Omega_M.
\end{equation}
For every \(\xi\in\partial_{\mathrm{RP}}\) and \(s\in\T\),
\(K_s(\Psi(\xi))=M_{|s|}^\xi(s)\). The central measure with marginals
\(M_n^\xi\) is extremal, and there are trees \(t_N\in\T_N\) with
\(t_N\nearrow t_{N+1}\) and \(t_N\to\Psi(\xi)\) in \(\Omega_M\).
See \zcref{thm:nested-paintbox-limit,thm:complete-boundary}.

\medskip

In the rest of the introduction we identify each sampling class
\(\xi\) with its Martin boundary point \(\Psi(\xi)\). Let \(\mathscr C\)
be the compact convex set of central measures. Write \(\mu^\xi\) for
the extremal measure associated with \(\xi\) and \(X_N\) for the tree
of size \(N\) along a path.

\noindent\textbf{Corollary B (Canonical decomposition).}
There is a Borel map \(\Xi\) from path space to \(\partial_{\mathrm{RP}}\)
such that, under every \(\mu\in\mathscr C\), \(X_N\to\Xi\) almost surely
in \(\Omega_M\). The tail \(\sigma\)-fields of the shape process and the
full path both coincide with \(\sigma(\Xi)\) modulo \(\mu\)-null sets.
The representation
\[
    \mu=\int_{\partial_{\mathrm{RP}}}\mu^\xi\,\Theta_\mu(d\xi),
    \qquad \Theta_\mu=\mathcal L_\mu(\Xi),
\]
is unique, and \(\mathscr C\) is a Bauer simplex.
See \zcref{thm:universal-boundary-variable,cor:bauer-simplex}.

\medskip

We now recover the root branch masses and their associated boundary points from the sampling laws. For each branch of positive mass \(p_i\), we record the pair \((p_i,z_i)\), where \(z_i\) is the boundary point represented by the paintbox on that branch. Replacing box \(i\) by dust can affect the sampled tree only if at least two labels fall in that box. This motivates assigning \((p_i,z_i)\) the weight \(p_i^2\), the probability that two given non-root labels both fall in box \(i\) (see \zcref{rem:dust-fixed-point-topology}).

For a nonempty compact metrizable space \(K\), let \(\mathfrak M(K)\) denote the space of finite measures on \([0,1]\times K\) of the form \(\sum_i p_i^2\delta_{(p_i,z_i)}\), where \(p_i>0\), \(z_i\in K\), and \(\sum_i p_i\le 1\). The sum may be empty, and repeated pairs are counted with multiplicity. We equip \(\mathfrak M(K)\) with the topology of weak convergence.

\noindent\textbf{Theorem C (Recursive root decomposition).}
For each boundary point \(\xi\), the sampling probabilities \(K_s(\xi)\)
for tree shapes \(s\) whose root has exactly one child determine the
multiset of pairs \((p_i,z_i)\) at the root of any paintbox representing
\(\xi\). These probabilities also characterize convergence on
\(\partial_{\mathrm{RP}}\) (\zcref{cor:root-decomposition-convergence}). The map
\begin{equation}\label{eq:intro-root-decomposition}
    \mathscr R:\partial_{\mathrm{RP}}\longrightarrow
    \mathfrak M(\partial_{\mathrm{RP}}),
    \qquad
    \mathscr R(\xi)=\sum_i p_i^2\delta_{(p_i,z_i)}
\end{equation}
is a homeomorphism (\zcref{thm:recursive-fixed-point-boundary}). Its inverse gives the sampling class of the paintbox
obtained by attaching a representative of each \(z_i\) as a branch of
mass \(p_i\) at a new root and assigning the remaining mass
\(1-\sum_i p_i\) to dust.

Associate to each paintbox an unordered rooted tree by retaining only the vertices reachable from the root along branches of positive mass. Give each edge the relative mass of the corresponding branch. Two paintboxes have the same finite sampling laws if and only if their associated trees are isomorphic as weighted rooted trees (\zcref{cor:paintbox-sampling-equivalence}).

For a finite tree \(t\) with at least two vertices, let
\(p_a=|a|/(|t|-1)\) be the relative size of a root branch \(a\).
The corresponding measure \(\nu_t=\sum_a p_a^2\delta_{(p_a,a)}\) marks
each branch by its image in \(\Omega_M\). If \(|t_n|\to\infty\), then
\(t_n\) converges to a boundary point \(\xi\) in \(\Omega_M\) if and only if
\(\nu_{t_n}\) converges weakly to \(\mathscr R(\xi)\) in
\(\mathfrak M(\Omega_M)\) (\zcref{cor:finite-tree-root-convergence}).

\medskip

We now identify the mixing measure in Corollary B for recursive Ewens
trees with parameters \(\theta,a>0\). Given a finite ordered set of labels, place the least
label at the root and partition the remaining labels according to the
Ewens law with parameter \(\theta a\). Conditional on this partition, construct the child subtrees independently
on their respective blocks, using the Ewens parameter \(\theta\)
throughout each subtree. The resulting shape laws, whose hook-product
formula is given in \zcref{sec:ewens-decomposition}, are coherent and
define a central measure \(\mu^{(\theta,a)}\). The choices
\((\theta,a)=(1,1)\) and \((2,1)\) give the uniform recursive-tree law
and Fulman's rooted-tree Plancherel law \cite{fulman2009mixing},
respectively.

\noindent\textbf{Theorem D (Poisson--Dirichlet decomposition).}
For \(\theta,a>0\), let \(\Omega\) be a recursive environment with
independent splits of law \(\PD(0,\theta a)\) at the root and
\(\PD(0,\theta)\) at every other vertex. At each vertex, list children
in decreasing order of mass. Let \(\xi\) be the sampling class of
\(\Omega\) and, conditionally on \(\Omega\), let \((X_N)\) be the shapes
of its nested samples. Then \(\mathcal L(\xi)\)
is the unique mixing measure of \(\mu^{(\theta,a)}\), and
\(X_N\to\xi\) almost surely in \(\Omega_M\).

Let \(\omega(t)\) be the recursive paintbox obtained from the ranked
subtree proportions of a finite tree \(t\), as defined before
\zcref{prop:finite-tree-paintbox-approximation}.
Then \(\omega(X_N)\to\Omega\) coordinatewise almost surely.
The ranked environment \(\Omega\) is almost surely a Borel function of
\(\xi\), and \(\sigma(\Omega)\) coincides with the tail
\(\sigma\)-field of \((X_N)\) modulo null sets under their joint law.
See \zcref{thm:ewens-decomposition,cor:pd-environment-recovery}.

The mixing measure of \(\mu^{(\theta,a)}\) has no atoms for every
\(\theta,a>0\) (\zcref{prop:ewens-not-extreme}). In particular,
the rooted-tree Plancherel law is not extremal. On the Young graph, the
Plancherel measure is extremal and corresponds to the point
\(\alpha=\beta=0\) of the Thoma simplex
\cite[Section~7]{kerov1994asymptotics}. For Thoma's theorem, see
\cite[Chapter~4]{borodin2017representations}.

Independent mass splits at the vertices also appear in earlier work on special cases of the recursive Ewens tree model. For \(a=1\), recording sibling birth order gives the
Chinese restaurant trees of \cite[Section~6.3]{evans2012trickle} with
\(\alpha=0\), whose boundary limits are described by independent
GEM splits \cite[Remark~6.4]{evans2012trickle}. When \(\theta=1\), our shape process is obtained by forgetting the labels in the \((0,a)\)-recursive tree of Dong, Goldschmidt and Martin, which they construct from independent Poisson--Dirichlet splits
\cite[Sections~5--6]{dong2006coagulation}
(see \zcref{ex:uniform-recursive-pd-specialization}). For
\((\theta,a)=(1,1)\), Janson showed that the uniform recursive-tree
process, viewed as a sequence of unordered trees, has the same law as a
split-tree process with Poisson--Dirichlet splits
\cite[Corollary~1.6]{janson2019split}.
In each case, Theorem D shows that the law of the sampling class of the
random environment is the unique mixing measure for the corresponding
central measure and that the ranked environment can be recovered from
the tail of the unordered shape process.

\subsection{Proof strategy}
For a finite tree \(t\in\T_N\), we compare the subtree induced by the first \(n\)
labels of a uniformly chosen increasing labelling with an \(n\)-vertex
sample from the recursive paintbox defined by the branch proportions of
\(t\). A coupling based on the hook formula bounds the total variation
distance between these shape laws by \(\binom{n-1}{2}/(N-1)\) for
\(2\le n\le N\). \zcref[S]{prop:finite-tree-paintbox-approximation}
shows that stars attain the largest distance over \(\T_N\) and that
this maximum tends to zero as \(N\to\infty\) if and only if
\(n=o(\sqrt N)\).

Compactness of the environment space and continuity of the sampling laws
then give a paintbox representation of every Martin boundary point.
Conversely, for a fixed paintbox, induction on the size of the test tree
shows that its nested samples converge almost surely to its sampling
class. This deterministic limit implies that the associated central
measure is extremal.

To recover the root masses and their marks, we use shapes of the form
\(\langle t\rangle\), obtained by adding a new root above \(t\).
In a sample with \(k\ge2\) non-root labels, the root has a single child
precisely when all \(k\) labels enter the same positive box. They all
enter box \(i\) with probability \(p_i^k\), in which case the child
subtree is sampled from \(z_i\). Summing over \(i\) gives, for
\(|s|\le k\),
\[
    \sum_{t\in\T_k} K_s(t)K_{\langle t\rangle}(\xi)
    =\sum_i p_i^k K_s(z_i).
\]
These moments determine the positive root masses and, for each such
mass \(p\), the measure \(\sum_{i:p_i=p}\mu^{z_i}\). Its image under
\(\Xi\) is \(\sum_{i:p_i=p}\delta_{z_i}\), which records the child
classes at mass \(p\) with their multiplicities.

\subsection{Organization}
\zcref[S]{sec:rooted,sec:recursive-paintboxes} develop the finite sampling
kernels and recursive paintboxes.
\zcref[S]{sec:recursive-paintbox-boundary} proves the boundary
classification and canonical decomposition,
\zcref{sec:intrinsic-root-structure} identifies the root data, and
\zcref{sec:ewens-decomposition} treats the Ewens family.

\subsection{Statement on AI use}
In preparing the first arXiv version of this paper, we used
OpenAI's GPT-5.6 Sol to fill in missing details in the proof
of the boundary classification. The proof in that version was
long and involved recursively matching descendant subtrees
and controlling the contribution of small branches by truncation.
We later used GPT-6 Astra to simplify the proof substantially,
in particular through the uniform approximation in
\zcref{prop:finite-tree-paintbox-approximation}.
The models were also used to improve the language and exposition
throughout the manuscript and to help identify
gaps and errors in the proofs.

\subsection*{Acknowledgments}
The author thanks his PhD supervisor, Pierre-Lo\"ic M\'eliot,
for many helpful discussions and valuable suggestions.

\subsection*{Conventions}

We write \(\N=\{1,2,\ldots\}\) and \(\Nzero=\{0,1,2,\ldots\}\).
For a compact space \(K\), write \(\mathcal P(K)\) for its Borel probability
measures, equipped with the weak topology.
For probability measures on a finite or countable set, we use
\(d_{\mathrm{TV}}(P,Q)=\frac12\sum_x|P(x)-Q(x)|\).

\section{Finite trees, paths, and sampling kernels}\label{sec:rooted}

\subsection{Rooted trees}

Let \(\T_n\) be the set of unlabelled non-plane rooted trees with \(n\)
vertices, considered up to rooted graph isomorphism (denoted by \(\cong\)), and let
\(\T:=\bigsqcup_{n\ge1}\T_n\). We index trees by their number of
vertices, so trees in \(\T_n\) have degree \(n-1\) in Hoffman's grading.
For each \(t\in\mathcal T\), fix a representative and denote its vertex set by \(V(t)\) and its root by \(\operatorname{root}(t)\). Write \(|t|=|V(t)|\). All shapes and counts
below are independent of this choice. For
\(x,v\in V(t)\), write \(x\preceq_t v\) when \(x\) lies on the
unique path from \(\rootv(t)\) to \(v\), allowing \(x=v\). For
\(v\in V(t)\), its hook length \(h_v(t)\) is the size of the descendant
subtree rooted at \(v\), including \(v\) itself. Let
\(\deg_t^+(v)\) be the number of children of \(v\). Write \(\Leaves(t)=\{v\in V(t):\deg_t^+(v)=0\}\) for its leaves,
and let \(\Aut(t)\) be its rooted-tree automorphism group. The one-vertex
tree is \(\one\). We write \(\langle t_1,\ldots,t_r\rangle\) for the
tree obtained by attaching \(t_1,\ldots,t_r\) below a new root, with
\(\langle\,\rangle=\one\). For \(m\ge1\), let \(C_m\) be the
\(m\)-vertex chain rooted at an endpoint and \(S_m\) the \(m\)-vertex
star rooted at its centre. In particular, \(C_1=S_1=\one\).

For \(t\in\T_n\) and a finite totally ordered set \(B\) with \(|B|=n\), an
\emph{increasing labelling of \(t\) by \(B\)} is a bijection
\(L:V(t)\to B\) that strictly increases along every root-to-leaf path.
Write \(\SL_B(t)\) for the set of these labellings. For the canonical
label set \(B_n:=\{0,1,\ldots,n-1\}\), write \(\SL(t):=\SL_{B_n}(t)\) and
\(d(t):=|\SL(t)|\).
The rooted automorphism group acts on \(\SL_B(t)\) by precomposition,
\((g\cdot L)(v):=L(g^{-1}v)\) for \(g\in\Aut(t)\) and \(v\in V(t)\).
The action is free: an automorphism that preserves a labelling fixes every
vertex. Define
\(\QSL_B(t):=\SL_B(t)/\Aut(t)\) and \(\QSL(t):=\QSL_{B_n}(t)\),
and call their elements \emph{quotient increasing labellings}. For the
canonical set \(B_n\), we use the shorter terms \emph{standard labelling} and
\emph{quotient standard labelling}. Since the action is free,
\(u(t):=|\QSL(t)|=d(t)/|\Aut(t)|\).

The increasing bijection \(B\to B_n\) identifies \(\SL_B(t)\) with
\(\SL(t)\) and \(\QSL_B(t)\) with \(\QSL(t)\). In particular,
\(|\SL_B(t)|=d(t)\) and \(|\QSL_B(t)|=u(t)\). We omit the
subscript \(B\) when the ordered label set is clear.

The number of standard labellings is given by the rooted-tree hook formula
\begin{equation}\label{eq:rooted-hook}
    d(t)=\frac{|t|!}{\prod_{v\in V(t)}h_v(t)}.
\end{equation}
Indeed, let \(t_1,\ldots,t_r\) be the root subtrees of a representative,
of sizes \(n_1,\ldots,n_r\), with isomorphic subtrees counted separately.
The root carries label \(0\). We distribute the remaining labels among the subtrees, assigning \(n_i\) labels to \(t_i\), and then choose an increasing labelling of each subtree. This gives
\[
    d(t)=\binom{|t|-1}{n_1,\ldots,n_r}\prod_{i=1}^r d(t_i),
\]
and \eqref{eq:rooted-hook} follows by induction (see
\cite{bergeron1992varieties} and
\cite[Section~2, equation~(4)]{sagan2008probabilistic}).

\subsection{Growth and pruning multiplicities}

For \(s\in\T_n\), \(t\in\T_{n+1}\), and \(v\in V(s)\), let \(s^{+v}\) be the tree
obtained by attaching a new leaf at \(v\), and write \(s\nearrow t\)
when \(s^{+v}\cong t\) for some \(v\). Hoffman's leaf-grafting and
leaf-removal multiplicities are, respectively,
\(n(s;t):=\#\{v\in V(s):s^{+v}\cong t\}\) and
\(m(s;t):=\#\{\ell\in\Leaves(t):t-\ell\cong s\}\)
\cite[Section~2]{hoffman2003combinatorics}. Both depend only on the rooted
isomorphism classes of \(s\) and \(t\). They need not be equal: for
\(s=C_2\) and \(t=S_3\), one has \(n(s;t)=1\) and \(m(s;t)=2\).

For \(s,t\in\T\) with \(k=|t|-|s|\ge0\), define
\begin{align}
    n(s;t)
    &:=\sum_{s=s_0\nearrow\cdots\nearrow s_k=t}
       \prod_{i=0}^{k-1}n(s_i;s_{i+1}),
       \label{eq:interval-grafting-path-count}\\
    m(s;t)
    &:=\sum_{s=s_0\nearrow\cdots\nearrow s_k=t}
       \prod_{i=0}^{k-1}m(s_i;s_{i+1}),
       \label{eq:interval-removal-path-count}
\end{align}
with \(n(t;t)=m(t;t)=1\).

Hoffman's symmetry factor is the product, over all vertices, of the
factorials of the multiplicities of isomorphic child subtrees, and hence
equals the order of the rooted automorphism group. His symmetry
relation \cite[Proposition~2.5(1)]{hoffman2003combinatorics} therefore takes
the following form.

\begin{proposition}[Hoffman symmetry relation]
\label{prop:hoffman-relation}
For all \(s,t\in\T\) with \(|s|\le |t|\),
\begin{equation}\label{eq:aut-relation}
    |\Aut(s)|\,m(s;t)=n(s;t)\,|\Aut(t)|.
\end{equation}
\end{proposition}
\begin{proof}
For each sequence of tree shapes \(s=s_0\nearrow\cdots\nearrow s_k=t\),
apply Hoffman's symmetry identity to each pair \((s_i,s_{i+1})\).
Multiplying the resulting identities and cancelling the intermediate
automorphism factors gives
\[
    |\Aut(s)|\prod_{i=0}^{k-1}m(s_i;s_{i+1})
    =\left(\prod_{i=0}^{k-1}n(s_i;s_{i+1})\right)|\Aut(t)|.
\]
Summing over these sequences yields
\eqref{eq:aut-relation}.
\end{proof}

\subsection{Dimensions and paths}

In a graded multigraph with a unique initial vertex, the dimension of a
vertex is the number of paths from the initial vertex to it, with parallel
edges counted separately. The leaf-grafting and leaf-removal multiplicities
define two graded multigraphs on the vertex set \(\T\), with upward edge
multiplicities \(n(s;t)\) and \(m(s;t)\), respectively. In both graphs, the initial vertex is the one-vertex tree \(\bullet\).

This pair of graphs fits Fomin's theory of duality for graded graphs
\cite{fomin1994duality}, as explained in
\cite[Section~2]{hoffman2003combinatorics}. The associated Plancherel
measure is discussed in \zcref{prop:fulman-plancherel-agreement}.

\begin{proposition}[Dimensions of the two Hoffman graphs]\label{prop:dimensions}
For every finite rooted tree \(t\), its dimension in the leaf-grafting
graph is \(u(t)\), and its dimension in the leaf-removal graph is \(d(t)\).
\end{proposition}

\begin{proof}
The largest label lies at a leaf, and deleting it from a standard
labelling of \(t\) gives
\[
    d(t)=\sum_{s:s\nearrow t}m(s;t)d(s),\qquad d(\one)=1,
\]
so \(d(t)\) is the dimension of \(t\) in the leaf-removal graph.
By \zcref{prop:hoffman-relation},
\[
    \sum_{s:s\nearrow t}n(s;t)u(s)
      =\frac1{|\Aut(t)|}\sum_{s:s\nearrow t}m(s;t)d(s)
      =\frac{d(t)}{|\Aut(t)|}=u(t).
\]
Together with \(u(\one)=1\), this identifies \(u(t)\) as the dimension
of \(t\) in the leaf-grafting graph.
\end{proof}

\zcref[S]{fig:hoffman-graph} shows the first four levels of the
leaf-grafting graph.

\begin{figure}[ht]
\centering
\begin{tikzpicture}[x=1cm,y=1cm,
    glyph/.style={circle,fill=black,draw=none,inner sep=0pt,
                  minimum size=3.6pt},
    tree edge/.style={draw=black,line width=0.5pt,line cap=round},
    graph edge/.style={draw=black!50,line width=0.4pt,line cap=round},
    multiplicity/.style={font=\scriptsize,text=black,fill=white,
                         inner sep=1pt}]
  \node[glyph] at (0,0) {};

  \begin{scope}[shift={(0,-0.95)},x=0.33cm,y=0.33cm]
    \draw[tree edge] (0,0)--(0,-1);
    \foreach \p in {(0,0),(0,-1)} \node[glyph] at \p {};
  \end{scope}

  \begin{scope}[shift={(-1.25,-2.10)},x=0.33cm,y=0.33cm]
    \draw[tree edge] (0,0)--(0,-1)--(0,-2);
    \foreach \p in {(0,0),(0,-1),(0,-2)} \node[glyph] at \p {};
  \end{scope}
  \begin{scope}[shift={(1.25,-2.10)},x=0.33cm,y=0.33cm]
    \draw[tree edge] (0,0)--(-0.6,-1) (0,0)--(0.6,-1);
    \foreach \p in {(0,0),(-0.6,-1),(0.6,-1)} \node[glyph] at \p {};
  \end{scope}

  \begin{scope}[shift={(-3.75,-3.65)},x=0.33cm,y=0.33cm]
    \draw[tree edge] (0,0)--(0,-1)--(0,-2)--(0,-3);
    \foreach \p in {(0,0),(0,-1),(0,-2),(0,-3)} \node[glyph] at \p {};
  \end{scope}
  \begin{scope}[shift={(-1.25,-3.65)},x=0.33cm,y=0.33cm]
    \draw[tree edge] (0,0)--(0,-1)--(-0.6,-2) (0,-1)--(0.6,-2);
    \foreach \p in {(0,0),(0,-1),(-0.6,-2),(0.6,-2)}
      \node[glyph] at \p {};
  \end{scope}
  \begin{scope}[shift={(1.25,-3.65)},x=0.33cm,y=0.33cm]
    \draw[tree edge] (0,0)--(-0.6,-1)--(-0.6,-2) (0,0)--(0.6,-1);
    \foreach \p in {(0,0),(-0.6,-1),(-0.6,-2),(0.6,-1)}
      \node[glyph] at \p {};
  \end{scope}
  \begin{scope}[shift={(3.75,-3.65)},x=0.33cm,y=0.33cm]
    \draw[tree edge] (0,0)--(-1,-1) (0,0)--(0,-1) (0,0)--(1,-1);
    \foreach \p in {(0,0),(-1,-1),(0,-1),(1,-1)}
      \node[glyph] at \p {};
  \end{scope}

  \draw[graph edge] (0,-0.14)--(0,-0.81);
  \draw[graph edge] (0,-1.43)--(-1.25,-1.96);
  \draw[graph edge] (0,-1.43)--(1.25,-1.96);
  \draw[graph edge] (-1.25,-2.91)--(-3.75,-3.51);
  \draw[graph edge] (-1.25,-2.91)--(-1.25,-3.51);
  \draw[graph edge] (-1.25,-2.91)--(1.25,-3.51);
  \draw[graph edge] (1.25,-2.58)--(1.25,-3.51)
    node[multiplicity,pos=0.48] {2};
  \draw[graph edge] (1.25,-2.58)--(3.75,-3.51);

  \node[fig note,font=\footnotesize,anchor=east] at (-4.55,0) {\(\T_1\)};
  \node[fig note,font=\footnotesize,anchor=east] at (-4.55,-1.115) {\(\T_2\)};
  \node[fig note,font=\footnotesize,anchor=east] at (-4.55,-2.43) {\(\T_3\)};
  \node[fig note,font=\footnotesize,anchor=east] at (-4.55,-4.145) {\(\T_4\)};
  \path (-5.05,0) (5.05,0);
\end{tikzpicture}
\caption{The first four levels of the leaf-grafting graph, with roots
at the top. Grey lines join trees related by the addition of one leaf.
Edge multiplicities \(n(s;t)\) are shown when they are greater than one.}
\label{fig:hoffman-graph}
\end{figure}

Since labels increase away from the root, the vertices with the \(k\) smallest labels form an ancestor-closed subtree. Any label-preserving isomorphism maps one such subtree onto the other, so restriction is well defined on quotient labellings.
\begin{proposition}[Paths and quotient standard labellings]
\label{prop:path-quotient-labelling-bijection}
For every \(n\ge1\) and \(t\in\T_n\), there is a bijection between
\(\QSL(t)\) and the paths from \(\one\) to \(t\) in the leaf-grafting graph.
These bijections can be chosen simultaneously so that restriction to
\(\{0,\ldots,k-1\}\) corresponds to truncation to the first \(k-1\)
edges, for every \(1\le k\le n\).

For each cover \(s\nearrow t\), every element of \(\QSL(s)\) has
\(n(s;t)\) extensions in \(\QSL(t)\). Hence \(u(s)n(s;t)\)
elements of \(\QSL(t)\) have predecessor shape \(s\) after deletion of the
maximal label.
\end{proposition}

\begin{proof}
Fix a cover \(s\nearrow t\) with \(|s|=n\), and choose a representative \(L\) of a quotient labelling \([L]\) of \(s\). For each vertex \(v\) with \(s^{+v}\cong t\), attaching a leaf labelled \(n\) at \(v\) gives an extension of \([L]\) to \(t\). To recover the attachment vertex from any extension, delete the leaf labelled \(n\). The remaining labelled tree has a unique label-preserving isomorphism to \((s,L)\), since all labels are distinct. The image of the deleted leaf’s parent under this isomorphism is the attachment vertex in \(s\).

Thus the extensions of \([L]\) to \(t\) are in bijection with the vertices \(v\) for which \(s^{+v}\cong t\), and there are \(n(s;t)\) of them. Since deleting the largest label determines a unique predecessor class, summing over the \(u(s)\) classes on \(s\) gives \(u(s)n(s;t)\) quotient labellings of \(t\) with predecessor shape \(s\).

For each cover \(s\nearrow t\) and each class \([L]\) on \(s\), fix a bijection between its extensions to \(t\) and the edges from \(s\) to \(t\). Starting with the unique labelling of the one-vertex tree, we use each successive edge of a path to select the next extension. Conversely, a quotient labelling and its predecessor determine the last edge through the chosen bijection. By induction, this construction gives a bijection between the paths from the one-vertex tree to \(t\) and the quotient standard labellings of \(t\), for every finite tree \(t\). Restricting a labelling to its \(k\) smallest labels corresponds to keeping the first \(k-1\) edges of the associated path.
\end{proof}

Throughout the paper, we fix one such compatible family of bijections.

\subsection{Finite sampling kernels}
\label{sec:sampling-algebra}

We now use the path counts from the preceding section to compute the
distribution of the subtree shape induced by the \(n\) smallest labels
in a uniformly chosen quotient standard labelling.
\begin{proposition}[The Hoffman sampling kernel]
\label{prop:sampling-kernel}
Fix trees \(s,t\) with \(|s|\le |t|\). 
The leaf-grafting and leaf-removal graphs have the
same kernels when normalized as sampling probabilities:
\begin{equation}\label{eq:Kn}
    K_s(t):=\frac{u(s)n(s;t)}{u(t)}=\frac{d(s)m(s;t)}{d(t)}.
\end{equation}
Set \(K_s(t)=0\) when \(|t|<|s|\). If \(t\in\T_N\) and \(L\) is chosen
uniformly from its quotient standard labellings, then for \(n\le N\) the
shape induced by labels \(0,\ldots,n-1\) has distribution
\((K_s(t))_{s\in\T_n}\).
\end{proposition}
\begin{proof}
The equality follows from \zcref{prop:hoffman-relation} and
\(d(r)=|\Aut(r)|u(r)\). By \zcref{prop:path-quotient-labelling-bijection}, a uniformly chosen
quotient standard labelling of \(t\) corresponds to a uniformly chosen
path from \(\one\) to \(t\) in the leaf-grafting graph. For \(s\in\T_n\),
the subtree induced by the \(n\) smallest labels has shape \(s\) precisely
when this path passes through \(s\). There are \(u(s)\) paths to \(s\)
and \(n(s;t)\) ways to continue from \(s\) to \(t\), so this event has
probability \(u(s)n(s;t)/u(t)=K_s(t)\).
\end{proof}

\begin{lemma}[Finite-level identities]\label{lem:finite-level-sampling-identities}
For \(t\in\T_N\),
\begin{equation}\label{eq:finite-kernel-identities}
\begin{aligned}
    \sum_{s\in\T_n}K_s(t)&=\ind_{\{n\le N\}},\\
    \frac{K_s(t)}{u(s)}
       &=\sum_{T:s\nearrow T}n(s;T)\frac{K_T(t)}{u(T)}
          \qquad(|s|<N).
\end{aligned}
\end{equation}
More generally, for \(|s|\le k\le N\),
\begin{equation}\label{eq:sampling-kernel-composition}
    K_s(t)=\sum_{r\in\T_k}K_s(r)K_r(t).
\end{equation}
\end{lemma}
\begin{proof}
For \(n\le N\), \zcref{prop:sampling-kernel} gives
\(\sum_{s\in\T_n}K_s(t)=1\). For \(n>N\), each term is zero by definition.
Decomposing each path from \(s\) to \(t\) at level \(k\) gives
\(n(s;t)=\sum_{r\in\T_k}n(s;r)n(r;t)\). Multiplication by
\(u(s)/u(t)\) proves \eqref{eq:sampling-kernel-composition}. The case
\(k=|s|+1\), divided by \(u(s)\), gives the second identity in
\eqref{eq:finite-kernel-identities}.
\end{proof}

\begin{example}[Uniform recursive trees]\label{ex:uniform-recursive-chain}
Attach each new vertex to a uniformly chosen existing vertex and forget the
arrival labels and sibling birth order. For \(s\nearrow t\), the resulting
Markov chain has transition probability \(P(s,t)=n(s;t)/|s|\), with
\(P(s,t)=0\) otherwise. Every rooted tree on the vertex set
\(\{0,\ldots,N-1\}\) whose labels increase away from the root
has probability \(1/(N-1)!\). Thus the unlabelled rooted shape has distribution
\begin{equation}\label{eq:uniform-recursive-marginal}
    M_N^{\mathrm{UR}}(t)=\frac{u(t)}{(N-1)!}.
\end{equation}
For \(s\in\T_n\) and \(t\in\T_N\) with \(n\le N\), we have
\(P^{N-n}(s,t)=\frac{(n-1)!}{(N-1)!}\,n(s;t)\).
Since the chain visits each level exactly once, its Doob--Martin kernel is
\[
    \frac{G(s,t)}{G(\one,t)}
    =\frac{P^{N-n}(s,t)}{M_N^{\mathrm{UR}}(t)}
    =\frac{K_s(t)}{M_n^{\mathrm{UR}}(s)},
\]
where \(G\) is the Green function. For \(N=n+1\),
\(K_s(t)\) is the backward transition probability from \(t\) to \(s\).
\end{example}

\begin{example}[Forgetting sibling birth order]
\label{ex:ordered-unordered-history}
Let \(t=\langle C_2,\one\rangle\), shown on the fourth level of
\zcref{fig:hoffman-graph}. There are three quotient standard labellings: the
singleton has label \(1\), \(2\), or \(3\), and the two remaining labels
increase along the chain. The first three labels form a chain only in the
last case, so \(K_{C_3}(t)=1/3\).
If sibling birth order is retained, \(t\) has two ordered versions.
When the long branch appears first, exactly one of the two labellings
compatible with this order gives a chain on the first three labels,
so the probability is \(1/2\).
When the singleton appears first, this probability is zero.
Forgetting birth order thus changes the sampling kernel. The birth-ordered
Harris representation in \cite{grubel2015recursive} keeps this
information.
\end{example}

\subsection{Central measures, harmonic functions, and tails}
\label{subsec:central-measures-harmonic}
To connect finite sampling with probability measures on infinite paths,
we recall the correspondence between central measures and nonnegative harmonic functions
on a Bratteli diagram, together with the tail criterion for extremality
\cite{bratteli1972inductive,effros1981dimensions,dynkin1978sufficient,vershik2014central}.
Further background can be found in \cite{kerov2003asymptotic} and
\cite[Chapter~7]{borodin2017representations}.

Let \(\Gamma=\bigsqcup_{n\ge1}\Gamma_n\) be a graded graph with finite
levels and \(\Gamma_1=\{\one\}\). Assume that every vertex is reachable
from \(\one\) and has at least one outgoing edge. For \(x\in\Gamma_n\)
and \(y\in\Gamma_{n+1}\), write \(\kappa(x,y)\) for the number of edges
from \(x\) to \(y\), which we assume is finite. As in the Hoffman graph,
paths distinguish parallel edges. The space \(\mathcal X_\Gamma\) of
infinite paths from \(\one\) is a closed subspace of the product of the finite sets of
edge copies between successive levels, hence compact metrizable.
Write \(X_n\) for the level-\(n\) vertex of a path,
\(\mathsf E_n\) for the edge copy leaving it, and \(\mathsf H_n(x)\) for
the set of finite paths from \(\one\) to \(x\in\Gamma_n\).
Set \(\dimf(x)=|\mathsf H_n(x)|\). For arbitrary vertices \(x\) and \(y\),
let \(\kappa(x,y)\) denote the number of directed paths from \(x\) to \(y\),
with \(\kappa(x,x)=1\).

For a finite path \(\gamma\) from \(\one\), let \(C_\gamma\) be the set
of infinite paths that begin with \(\gamma\).
Write
\begin{equation}\label{eq:raw-future-field-definition}
\begin{split}
    \mathscr F_n&=\sigma(\mathsf E_1,\ldots,\mathsf E_{n-1}),\\
    \mathscr G_N&=\sigma(X_N,\mathsf E_N,\mathsf E_{N+1},\ldots),
    \qquad \mathscr G_\infty=\bigcap_{N\ge1}\mathscr G_N.
\end{split}
\end{equation}
These are the uncompleted past, future, and tail \(\sigma\)-fields.
For a probability measure \(\mu\), write \(\overline{\mathscr H}^{\,\mu}\)
for \(\mathscr H\) augmented by all subsets of Borel \(\mu\)-null sets.
For a tail field, we take the intersection before completing it.

A probability measure \(\mu\) on \(\mathcal X_\Gamma\), with marginals
\(M_n^\mu(x)=\mu(X_n=x)\), is \emph{central} if the probability of a cylinder \(C_\gamma\) depends only on the endpoint of \(\gamma\),
that is,
\begin{equation}\label{eq:central-cylinder-definition}
    \mu(C_\gamma)=\frac{M_n^\mu(x)}{\dimf(x)},
    \qquad\gamma\in\mathsf H_n(x).
\end{equation}
The path cylinders form a countable clopen base, so their probabilities
determine a probability measure and characterize weak convergence. It follows from \eqref{eq:central-cylinder-definition} that the set
\(\mathscr C(\Gamma)\) of central measures is a closed convex subset
of the compact space \(\mathcal P(\mathcal X_\Gamma)\).

The marginals \(M_n=M_n^\mu\) satisfy the coherence relation
\begin{equation}\label{eq:coherent-family-definition}
    M_n(x)=\sum_{y:x\nearrow y}
      M_{n+1}(y)\frac{\kappa(x,y)\dimf(x)}{\dimf(y)}.
\end{equation}
Conversely, for any family of probability measures \((M_n)\) satisfying
\eqref{eq:coherent-family-definition}, the cylinder probabilities
\(M_n(x)/\dimf(x)\) are consistent. They therefore extend to a unique
central measure with these marginals. Dividing the coherence relation by
\(\dimf(x)\) identifies central measures with normalized nonnegative
harmonic functions:
\begin{equation}\label{eq:harmonic-general}
    \varphi(x)=\frac{M_n(x)}{\dimf(x)},\qquad
    \varphi(x)=\sum_{y:x\nearrow y}\kappa(x,y)\varphi(y),
    \qquad \varphi(\one)=1.
\end{equation}
Let \(\gamma\) be a finite path ending at \(x\in\Gamma_n\) with
\(\mu(C_\gamma)>0\). For each edge \(e\) from \(x\) to \(y\), the cylinder
formula gives \(\mu(C_{\gamma e})/\mu(C_\gamma)=\varphi(y)/\varphi(x)\).
Summing over the \(\kappa(x,y)\) edges shows that \((X_n)\) is Markov.
For \(x\nearrow y\), its forward and backward transition probabilities are
\begin{align}
    \Prob_\mu(X_{n+1}=y\mid X_n=x)
       &=\kappa(x,y)\frac{\varphi(y)}{\varphi(x)}
         &&\text{if }M_n(x)>0, \label{eq:central-forward-transition}\\
    \Prob_\mu(X_n=x\mid X_{n+1}=y)
       &=\kappa(x,y)\frac{\dimf(x)}{\dimf(y)}
         &&\text{if }M_{n+1}(y)>0. \notag
\end{align}

The function \(\varphi\) spans a \emph{minimal ray} if every nonnegative
harmonic function \(\psi\le\varphi\) is a scalar multiple of \(\varphi\).

For a central measure \(\mu\) and a cylinder \(C_\gamma\), the sequence
\(\Prob_\mu(C_\gamma\mid\mathscr G_N)\) is a reverse martingale.
The limits of these martingales give the following extremality criterion.

\begin{proposition}[Endpoint kernels, tails, and extremality]
\label{prop:central-kernels-tail}
Let \(\mu\) be central. For \(x\in\Gamma_n\) and \(y\in\Gamma_N\) with
\(n\le N\), set \(K_x(y)=\dimf(x)\kappa(x,y)/\dimf(y)\).
\begin{enumerate}[label=(\roman*)]
\item For \(x\in\Gamma_n\), \(N\ge n\), and
\(\gamma\in\mathsf H_n(x)\),
\begin{align}
    \Prob_\mu(C_\gamma\mid\mathscr G_N)
      &=\frac{K_x(X_N)}{\dimf(x)},\label{eq:path-martin-general}\\
    \Prob_\mu(X_n=x\mid\mathscr G_N)
      &=K_x(X_N).\label{eq:endpoint-martin-general}
\end{align}
Consequently \(K_x(X_N)\) converges almost surely and in \(L^1\) to
\(\Prob_\mu(X_n=x\mid\mathscr G_\infty)\).
\item The measure \(\mu\) is extremal if and only if
\(\mathscr G_\infty\) is trivial, or equivalently if its normalized
harmonic function spans a minimal ray. These conditions are also
equivalent to the simultaneous limits
\begin{equation}\label{eq:martin-kernel-extreme-general}
    K_x(X_N)\longrightarrow M_n^\mu(x)
    \quad\mu\text{-almost surely},\qquad n\ge1,\ x\in\Gamma_n.
\end{equation}
\end{enumerate}
\end{proposition}
\begin{proof}
For (i), fix a path \((e_N,\ldots,e_{M-1})\) from \(y\in\Gamma_N\)
to \(z\in\Gamma_M\), and let
\(F=\{X_N=y,\ \mathsf E_j=e_j\text{ for }N\le j<M\}\).
Centrality and path counting give
\[
    \mu(C_\gamma\cap F)=\kappa(x,y)\frac{M_M^\mu(z)}{\dimf(z)},
    \qquad
    \mu(F)=\dimf(y)\frac{M_M^\mu(z)}{\dimf(z)}.
\]
Since \(X_N=y\) on \(F\), these identities give
\(\mu(C_\gamma\cap F)=\dimf(x)^{-1}\int_F K_x(X_N)\,d\mu\).
Such cylinders generate \(\mathscr G_N\), so the identity extends to all
\(F\in\mathscr G_N\) by the monotone-class theorem. This proves
\eqref{eq:path-martin-general}. Summing over \(\gamma\in\mathsf H_n(x)\)
gives \eqref{eq:endpoint-martin-general}.
The convergence follows from the reverse-martingale theorem.

For (ii), suppose that \(\mu\) is extremal and fix \(B\in\mathscr G_\infty\). Integrating
\eqref{eq:path-martin-general} and \eqref{eq:endpoint-martin-general}
over \(B\) gives
\begin{equation}\label{eq:central-tail-prefix-uniformity}
    \mu(C_\gamma\cap B)
       =\frac{\mu(\{X_n=x\}\cap B)}{\dimf(x)}.
\end{equation}
If \(0<\mu(B)<1\), conditioning on \(B\) and its complement decomposes
\(\mu\) into distinct central measures. Conversely, suppose that \(\mathscr G_\infty\) is trivial and write
\(\mu=b\nu_1+(1-b)\nu_2\), where \(0<b<1\) and both measures are central.
For \(H=d\nu_1/d\mu\), centrality gives
\begin{equation}\label{eq:rn-prefix-martingale-central}
    \E_\mu[H\mid\mathscr F_n]
       =\frac{M_n^{\nu_1}(X_n)}{M_n^\mu(X_n)}=:H_n,
\end{equation}
with \(0/0:=0\). Since the fields \(\mathscr F_n\) generate the Borel field,
\(H_n\to H\) almost surely. For every fixed
\(N\), the variables \(H_n\), \(n\ge N\), are
\(\mathscr G_N\)-measurable, so \(\limsup_n H_n\) is a
\(\mathscr G_\infty\)-measurable version of \(H\). Since \(\mathscr G_\infty\) is trivial, \(H\) is almost surely constant
and hence equal to \(\E_\mu[H]=1\).
Thus \(\nu_1=\mu\), and the decomposition of \(\mu\) gives \(\nu_2=\mu\) as well.

Tail triviality gives \eqref{eq:martin-kernel-extreme-general} by (i).
Conversely, these limits and \eqref{eq:path-martin-general} imply
\(\Prob_\mu(C_\gamma\mid\mathscr G_\infty)=\mu(C_\gamma)\) for every
cylinder. The monotone-class theorem extends this to all Borel path events,
so the tail is independent of the full path field and is therefore trivial.

Finally, a nonzero nonnegative harmonic function is positive at \(\one\),
since every vertex is reachable from \(\one\). Thus a decomposition of
\(\varphi\) into two nonzero, nonproportional nonnegative harmonic functions gives
a nontrivial convex decomposition after normalizing each summand at
\(\one\), and conversely. By \eqref{eq:harmonic-general}, \(\mu\) is
therefore extremal if and only if \(\varphi\) spans a minimal ray.
\end{proof}

For the leaf-grafting Hoffman graph, \(\Gamma_n=\T_n\),
\(\kappa(s,t)=n(s;t)\), and \(\dimf(t)=u(t)\), so the endpoint
kernels above are the sampling kernels in \eqref{eq:Kn}.
Here \(X_n\) is the tree of size \(n\), and \(\mathsf H_n(t)\) consists
of the finite paths from \(\one\) to \(t\). We abbreviate
\(\mathscr C(\Gamma)\) to \(\mathscr C\).
The uniform recursive-tree law of \zcref{ex:uniform-recursive-chain}
corresponds to \(\varphi(s)=1/(|s|-1)!\).

\subsection{The compactification and minimal boundary}
\label{sec:martin-compactification}
A sequence \((t_j)\) of finite trees converges in the Martin
compactification if and only if \(K_s(t_j)\) converges for every
\(s\in\T\) (see \zcref{ex:uniform-recursive-chain}). We therefore set
\begin{equation}\label{eq:martin-coordinate-map}
    \iota_M(t)=(K_s(t))_{s\in\T},\qquad
    \Omega_M=\overline{\iota_M(\T)}\subseteq[0,1]^\T.
\end{equation}
The level sums in \eqref{eq:finite-kernel-identities} recover \(|t|\):
\(\sum_{s\in\T_n}K_s(t)=\ind_{\{n\le |t|\}}\). At level \(|t|\),
\(K_s(t)=\ind_{\{s=t\}}\), so the coordinates also recover \(t\).
Thus \(\iota_M\) is injective. We identify \(\T\) with its image in the
compact metrizable space \(\Omega_M\).

Similar sampling coordinates are used for the Martin boundaries of
erased-word processes \cite[Section~2.2]{gerstenberg2020words} and
erased-interval processes \cite{gerstenberg2020interval}.
Other related classifications include the Young--Fibonacci lattice
\cite{goodman2000martin}, whose boundary points are all minimal
\cite{bochkov2026ergodicity}, the zigzag graph
\cite{gnedin2006zigzag,tarrago2018zigzag}, R\'emy's chain
\cite{evans2017remy}, and radix sort trees \cite{evans2017radix}.

\begin{lemma}[Finite trees and boundary level sums]
\label{lem:finite-vertices-isolated}
Every finite tree is an isolated point of the Martin compactification.
The boundary
\(\partial\Omega_M=\Omega_M\setminus\T\) is compact and equals
\begin{equation}\label{eq:martin-boundary-level-sums}
    \left\{x\in\Omega_M:
       \sum_{s\in\T_n}K_s(x)=1\text{ for every }n\ge1\right\}.
\end{equation}
Every boundary point is approached by finite trees whose sizes tend to
infinity.
\end{lemma}
\begin{proof}
For \(t\in\T_N\), the open set
\(\{K_t>1/2,\ \sum_{r\in\T_{N+1}}K_r<1/2\}\) meets \(\T\) only at
\(t\). Since \(\mathcal T\) is dense in \(\Omega_M\), this open set equals \(\{t\}\).
Thus \(\mathcal T\) is open in \(\Omega_M\), so its complement \(\partial\Omega_M\) is compact. Fix \(x\in\partial\Omega_M\) and choose \(t_j\in\mathcal T\) with \(t_j\to x\). Since each level is finite, \(|t_j|\to\infty\). For every fixed \(n\),
the sum \(\sum_{s\in\T_n}K_s(t_j)\) is eventually one, so
\(\sum_{s\in\T_n}K_s(x)=1\). A finite tree \(t\) has level sum zero
at \(n=|t|+1\). This proves \eqref{eq:martin-boundary-level-sums}.
\end{proof}

For \(\zeta\in\partial\Omega_M\), set
\(\varphi_\zeta(s)=K_s(\zeta)/u(s)\).
The identities in \eqref{eq:finite-kernel-identities} pass to the limit and
give
\begin{equation}\label{eq:boundary-harmonicity}
    \varphi_\zeta(s)=\sum_{T:s\nearrow T}n(s;T)\varphi_\zeta(T),
    \qquad \varphi_\zeta(\one)=1.
\end{equation}
Thus \(M_n^\zeta(s)=K_s(\zeta)\), \(s\in\T_n\), defines a
coherent family of probability measures and hence a central measure. The \emph{minimal boundary}
\(\partial_{\min}\Omega_M\) consists of points for which this measure is
extremal, equivalently for which \(\varphi_\zeta\) spans a minimal harmonic
ray \cite[Chapter~7]{woess2009denumerable}.

\begin{proposition}[Extremal measures and the minimal boundary]
\label{prop:extremal-laws-minimal-boundary}
Extremal central measures are in bijection with \(\partial_{\min}\Omega_M\).
For an extremal measure \(\mu\), its point \(\zeta_\mu\) is characterized by
\(K_s(\zeta_\mu)=M_{|s|}^\mu(s)\).
\end{proposition}
\begin{proof}
By \zcref{prop:central-kernels-tail}, under an extremal measure \(\mu\),
\(K_s(X_N)\to M_{|s|}^{\mu}(s)\) simultaneously for all \(s\in\T\)
on one event of probability one. These limits define a point of \(\Omega_M\), and the level sums
show that it lies on the boundary (\zcref{lem:finite-vertices-isolated}). Its central measure is \(\mu\), so it is minimal.
Conversely, every minimal boundary point induces an extremal central measure by definition. If \(\zeta\) induces \(\mu\), then \(\zeta=(M^\mu_{|s|}(s))_{s\in\mathcal T}\), so \(\mu\) determines \(\zeta\) uniquely.
\end{proof}

\begin{figure}[ht]
\centering
\begin{tikzpicture}[x=1cm,y=1cm]
  \node[fig panel] at (0,0.42) {\(h_4\)};
  \node[fig vertex] (r) at (0,0) {};
  \foreach \i/\x in {1/-0.9,2/-0.3,3/0.3}
    {\node[fig vertex] (c\i) at (\x,-0.75) {};
     \draw[fig edge] (r)--(c\i);}
  \node[fig vertex,fill=figblue] (c4) at (0.9,-0.75) {};
  \draw[fig edge] (r)--(c4);
  \foreach \j/\y in {2/-1.4,3/-2.05,4/-2.7}
    {\node[fig vertex,fill=figblue] (d\j) at (0.9,\y) {};}
  \draw[fig edge,figblue] (c4)--(d2)--(d3)--(d4);
  \node[fig note,text=figblue,anchor=west] at (1.12,-1.725) {\(C_4\)};

  \begin{scope}[xshift=5.0cm]
    \node[fig panel] at (0,0.42) {\(s_4\)};
    \node[fig vertex] (q) at (0,0) {};
    \foreach \i in {1,...,7}
      {\node[fig vertex] (m\i) at ({-1.8+0.6*(\i-1)},-0.75) {};
       \draw[fig edge] (q)--(m\i);}
  \end{scope}
\end{tikzpicture}
\caption{The trees \(h_4\) and \(s_4\) of
\zcref{ex:ordered-boundary-nonextension}, with siblings ordered by birth
from left to right. In \(h_4\), the three singleton branches precede the
chain \(C_4\).}
\label{fig:ordered-example}
\end{figure}

\begin{example}[Birth order and boundary limits]
\label{ex:ordered-boundary-nonextension}
Let \(\mathcal H\) be the set of finite rooted trees whose siblings are
ordered by birth, and let \(\overline{\mathcal H}_M\) be the Martin
compactification of the uniform recursive-tree process with this order
retained. Its boundary is described in
\cite[Theorem~6.1]{evans2012trickle}, with
parameters \((\alpha,\theta)=(0,1)\) (see also
\cite[Section~2.3]{grubel2015recursive}). Write
\(\pi_{\mathrm{ord}}:\mathcal H\to\T\) for the map that forgets
sibling order. If \((Y_m)_{m\ge1}\) is the ordered growth chain, its
Martin coordinates, normalized as conditional probabilities, are
\[
    \mathcal K_b(g)
       =\Prob\{Y_m=b\mid Y_{|g|}=g\},
       \qquad |b|=m\le |g|.
\]
Set \(\mathcal K_b(g)=0\) when \(|b|>|g|\). 

For \(n\ge2\), let \(h_n\in\mathcal H\) have \(n\) root branches:
the first \(n-1\) are singletons and the last is the chain \(C_n\).
Let \(s_n\) be the birth-ordered star with \(2n\) vertices, so
\(|h_n|=|s_n|=2n\) (see \zcref{fig:ordered-example}). For each fixed
\(m\le n\), every path to \(h_n\) passes through the \(m\)-vertex star.
Indeed, the first \(n-1\) children must appear before the last branch,
and none of them has descendants. Every path to \(s_n\) also passes
through this star. Thus for \(n\ge |b|\),
\(\mathcal K_b(h_n)=\mathcal K_b(s_n)=\ind_{\{b\text{ is a star}\}}\).
Hence \((h_n)\) and \((s_n)\) have the same Martin limit in \(\overline{\mathcal H}_M\).

After forgetting order, the uniform allocation of labels to root branches gives
\begin{equation}\label{eq:ordered-unordered-boundary-separation}
    K_{C_3}(\pi_{\mathrm{ord}}(h_n))
       =\frac{n(n-1)}{(2n-1)(2n-2)}\longrightarrow\frac14,
    \qquad
    K_{C_3}(\pi_{\mathrm{ord}}(s_n))=0.
\end{equation}
The first probability is the chance that both non-root sample labels lie
in the long branch. Since \(K_{C_3}\) is continuous
on \(\Omega_M\), the map \(\pi_{\mathrm{ord}}\) has no continuous
extension from \(\overline{\mathcal H}_M\) to \(\Omega_M\).
\end{example}

\section{Recursive paintboxes and finite sampling}
\label{sec:recursive-paintboxes}
We use the split-tree construction with one label at each vertex
\cite{devroye1998universal,janson2019split}, taking the splits to be
deterministic and allowing dust. Related constructions include
nested occupancy schemes
\cite{gnedin2020nested} and recursive partition structures
\cite{gnedin2006recursive}. Exchangeable hierarchies and weighted real
trees \cite{forman2018representation,forman2020mass} provide another
sampling framework based on nested blocks. The sample labels lie at the
leaves of the associated finite hierarchy trees, as in the regenerative
tree growth models of \cite{pitman2014regenerative}.
Geldbach \cite{geldbach2026continuum} gives sampling representations for
tree-growth chains whose backward step deletes a uniformly chosen leaf
and suppresses degree-two vertices. In our construction, each vertex
carries a label, and the backward step deletes the maximal label of a
uniform quotient labelling.

We show that recursive paintbox sampling gives central measures that depend
continuously on the environment. We then derive a convergence criterion
for finite trees from their root decompositions.

\subsection{Recursive sampling}
\label{subsec:recursive-paintboxes}

Let \(\Ulam=\{\varnothing\}\cup\bigcup_{d\ge1}\N^d\) be the Ulam tree.
A word \(w=i_1\cdots i_d\) records the successive child indices along a
path from the root and is called the address of the corresponding vertex.
For words \(u,v\in\Ulam\), write \(uv\) for their
concatenation, with \(u\varnothing=\varnothing u=u\). An integer
\(i\in\N\) is also viewed as a one-letter word. Then \(wi\) is the \(i\)th child of \(w\), and \(iw\) is the vertex with relative address \(w\) in the subtree rooted at \(i\).
\begin{definition}[Recursive paintbox]
\label{def:ulam-recursive-paintbox-representative}
A deterministic recursive paintbox is an array
\(\omega=(p_{w,i}:w\in\Ulam,\ i\in\Nzero)\) such that, for every word
\(w\),
\[
    p_{w,i}\ge0\quad(i\ge1),
    \qquad
    \sum_{i\ge1}p_{w,i}\le1,
    \qquad
    p_{w,0}=1-\sum_{i\ge1}p_{w,i}.
\]
\end{definition}
For \(v\in\Ulam\), let \(\omega_v\) be the paintbox rooted at \(v\),
given by \((\omega_v)_{w,j}=p_{vw,j}\) for \(w\in\Ulam\) and
\(j\in\Nzero\). In particular, \(\omega_\varnothing=\omega\).
At the root, write \(p_i=p_{\varnothing,i}\) for \(i\ge0\). Then
\(p_0\) is the dust mass, and each child \(i\ge1\) with \(p_i>0\)
has mass \(p_i\) and carries the paintbox \(\omega_i\). We call the split at \(w\) \emph{pure dust} if \(p_{w,0}=1\).

At each address \(w\), we list the positive child masses in nonincreasing
order, breaking ties arbitrarily and padding finite lists with zeros.
Each child keeps its descendant environment under this reordering, so
the finite sampling laws are unchanged
(see \zcref{rem:paintbox-equivalence-relabelling}).

\begin{definition}[Sampling from a recursive paintbox]\label{def:recursive-paintbox-sampling}
Fix a deterministic recursive paintbox \(\omega\) with root masses \(p_i\),
child environments \(\omega_i\), and root dust mass \(p_0\).
For a finite nonempty totally ordered label set \(B\), we construct a
random labelled rooted tree \(\widetilde T_B(\omega)\) recursively.
\begin{enumerate}[label=(\roman*)]
    \item Place \(\min B\) at the root. If \(|B|=1\), sampling stops here.
    \item Each remaining label independently chooses child box
    \(i\ge1\) with probability \(p_i\), or dust with probability \(p_0\).
    \item Dust labels become leaf children of the root.
    \item Let \(B_i\) be the set of labels assigned to box \(i\). Given
    these sets, sample the subtrees \(\widetilde T_{B_i}(\omega_i)\)
    independently for the occupied boxes and attach their roots to the
    vertex labelled \(\min B\).
\end{enumerate}
The labels increase strictly along every root-to-leaf path because the
least label of each block becomes its root. Let \(T_B(\omega)\) be the
unlabelled shape of \(\widetilde T_B(\omega)\). We write
\(\widehat T_B(\omega)=(s,L)\) when \(T_B(\omega)=s\) and
\(\widetilde T_B(\omega)\) has quotient labelling \(L\in\QSL_B(s)\).
For the canonical label set \(B_n=\{0,\ldots,n-1\}\), we use
the subscript \(n\) in place of \(B_n\) and denote the shape law by
\(M_n^\omega(s)=\Prob(T_n(\omega)=s)\), \(s\in\T_n\).
\end{definition}

At each vertex, the remaining labels are allocated according to
Kingman's paintbox~\cite{kingman1978representation}.
\zcref[S]{fig:paintbox-sample}
shows a recursive paintbox with dust, the allocation of labels at its
root, and the resulting sample.

\begin{figure}[ht]
\centering
\begin{tikzpicture}[x=1cm,y=1cm]
  \node[fig env] (r) at (0,0) {};
  \node[fig env,draw=figblue] (a) at (-0.95,-1.05) {};
  \node[fig env,draw=figorange] (b) at (0.95,-1.05) {};
  \draw[fig edge,figblue]
    (r) -- node[fig mass,above left] {\(\tfrac12\)} (a);
  \draw[fig edge,figorange]
    (r) -- node[fig mass,above right] {\(\tfrac14\)} (b);
  \draw[fig dust]
    (r) -- node[fig mass,right,pos=0.55] {\(\tfrac14\)} (0,-0.86);
  \pic[scale=1.25] at (0,-0.98) {fig dustcloud};
  \node[fig note,below] at (0,-1.08) {dust};
  \node[fig env,draw=figblue] (a1) at (-0.95,-2.1) {};
  \node[fig env,draw=figblue] (a2) at (-0.95,-3.15) {};
  \draw[fig edge,figblue]
    (a) -- node[fig mass,left] {\(1\)} (a1)
    (a1) -- node[fig mass,left] {\(1\)} (a2);
  \node[text=figblue,inner sep=0pt] at (-0.95,-3.55) {\(\vdots\)};
  \node[fig note,left=2pt] at (a.west) {box 1};
  \draw[fig dust]
    (b) -- node[fig mass,right,pos=0.55] {\(1\)} (0.95,-1.91);
  \pic[scale=1.25] at (0.95,-2.03) {fig dustcloud};
  \node[fig note,below] at (0.95,-2.13) {dust};
  \node[fig note,right=2pt] at (b.east) {box 2};
  \node[fig panel] at (0,-4.2) {(a) Environment};

  \begin{scope}[xshift=4.35cm]
    \coordinate (O) at (-1.9,-2.05);
    \fill[gray!20]      (O) rectangle ++(0.95,0.32);
    \fill[figblue!18]   ($(O)+(0.95,0)$) rectangle ++(1.9,0.32);
    \fill[figorange!20] ($(O)+(2.85,0)$) rectangle ++(0.95,0.32);
    \draw[line width=0.5pt] (O) rectangle ++(3.8,0.32);
    \foreach \x in {0.95,2.85}
      \draw[line width=0.5pt] ($(O)+(\x,0)$) -- ++(0,0.32);
    \node[fig note,below=2.5pt] at ($(O)+(0.475,0)$) {dust};
    \node[fig note,below=2.5pt] at ($(O)+(1.9,0)$) {box 1};
    \node[fig note,below=2.5pt] at ($(O)+(3.325,0)$) {box 2};
    \node[fig mass,left=1.5pt] at ($(O)+(0,0.16)$) {\(0\)};
    \node[fig mass,right=1.5pt] at ($(O)+(3.8,0.16)$) {\(1\)};
    \foreach \lab/\u/\col in {2/0.11/figgray,1/0.29/figblue,
      3/0.45/figblue,6/0.62/figblue,4/0.80/figorange,5/0.94/figorange}{
      \node[fig label,minimum size=12pt,draw=\col,fill=\col!14] (u\lab)
        at ($(O)+({3.8*\u},1)$) {\(\lab\)};
      \draw[line width=0.45pt,black!65]
        (u\lab.south) -- ($(O)+({3.8*\u},0.32)$);
      \fill[black] ($(O)+({3.8*\u},0.32)$) circle (1pt);
    }
    \node[fig panel] at (0,-4.2) {(b) Root allocation};
  \end{scope}

  \begin{scope}[xshift=8.7cm]
    \node[fig label] (s0) at (0,0) {0};
    \node[fig label,draw=figblue,fill=figblue!12]
      (s1) at (-0.95,-1.05) {1};
    \node[fig label,draw=figgray,fill=gray!16]
      (s2) at (0,-1.05) {2};
    \node[fig label,draw=figorange,fill=figorange!14]
      (s4) at (0.95,-1.05) {4};
    \node[fig label,draw=figblue,fill=figblue!12]
      (s3) at (-0.95,-2.1) {3};
    \node[fig label,draw=figblue,fill=figblue!12]
      (s6) at (-0.95,-3.15) {6};
    \node[fig label,draw=figgray,fill=gray!16]
      (s5) at (0.95,-2.1) {5};
    \draw[fig edge,figblue] (s0)--(s1) (s1)--(s3) (s3)--(s6);
    \draw[fig edge,figorange] (s0)--(s4);
    \draw[fig edge,figgray] (s0)--(s2) (s4)--(s5);
    \node[fig note,left=2pt] at (s1.west) {box 1};
    \node[fig note,right=2pt] at (s4.east) {box 2};
    \node[fig panel] at (0,-4.2) {(c) Sample};
  \end{scope}
\end{tikzpicture}
\caption{(a) A recursive paintbox. Labels on solid edges give relative
child masses, and dotted lines indicate dust. Box~1 and its descendants
each have one child of mass one, whereas box~2 has pure dust.
(b) Each non-root label chooses an independent uniform point of
\([0,1)\). Labels in the dust interval become separate leaves at the root.
(c) The sample on \(\{0,\ldots,6\}\). Labels \(1,3,6\) form a chain
in box~1, and label \(2\) becomes a dust leaf at the root.
Labels \(4,5\) enter box~2, where \(4\) occupies the local root and
\(5\) becomes its dust leaf. Grey edges lead to dust leaves.}
\label{fig:paintbox-sample}
\end{figure}

\begin{lemma}[Consistency and relabelling]
\label{lem:paintbox-consistency}
Fix a deterministic recursive paintbox \(\omega\) and a finite nonempty
totally ordered label set \(B\).
\begin{enumerate}[label=(\roman*)]
\item The samples on all nonempty initial segments \(A\subseteq B\)
can be coupled using the same box choices for the labels in \(A\) so that
the restriction of \(\widetilde T_B(\omega)\) to \(A\) equals
\(\widetilde T_A(\omega)\).
\item If \(\phi:B\to B'\) is an increasing bijection, the tree obtained
by relabelling \(\widetilde T_B(\omega)\) by \(\phi\) has the same law as
\(\widetilde T_{B'}(\omega)\).
\end{enumerate}
In particular, the samples on \(B_n=\{0,\ldots,n-1\}\) can be coupled
simultaneously for all \(n\) so that
\(\widetilde T_{n+1}(\omega)\setminus\{n\}
=\widetilde T_n(\omega)\).
\end{lemma}

\begin{proof}
Let
\((U_{m,w})_{m\ge1,\,w\in\Ulam}\) be independent uniform variables
on \([0,1)\), and set
\begin{equation}\label{eq:uniform-routing-realization}
    J_{m,w}=\min\left\{i\in\Nzero:
                    U_{m,w}<\sum_{j=0}^{i}p_{w,j}\right\}.
\end{equation}
The minimum exists because \(\sum_{j\ge0}p_{w,j}=1\). These choices
are independent, with \(\Prob(J_{m,w}=i)=p_{w,i}\). 
Write the elements of \(B\) as \(b_0<\cdots<b_{|B|-1}\).
When \(b_m\) reaches an address \(w\) and is not the local root,
it chooses box \(J_{m,w}\), with \(0\) denoting dust.

Now let \(A\) be a nonempty initial segment of \(B\). Each label in \(A\)
keeps its index \(m\) and hence makes the same box choices in both samples.
Moreover, the least label of every block meeting \(A\) belongs to \(A\),
so the block retains its root. It follows recursively that
\(\widetilde T_B(\omega)|_A=\widetilde T_A(\omega)\), proving (i).
Using the same array for all \(B_n\) gives the simultaneous coupling.

For (ii), give each label and its image under \(\phi\) the same choices.
The bijection \(\phi\) sends each block of labels to the corresponding
block in the relabelled sample and preserves the order within it. Induction within these blocks shows that the
two labelled trees are related by \(\phi\).
\end{proof}

\subsection{Central measures and the sampling topology}

The recursive rule gives the same probability to all quotient increasing
labellings of a given shape, so every paintbox defines a central measure.

\begin{proposition}[Centrality and coherence]
\label{prop:paintbox-central}
Let \(\omega\) be a deterministic recursive paintbox. For every
\(n\ge1\), \(s\in\T_n\), and \(L\in\QSL(s)\),
\[
    \Prob_\omega\{\widehat T_n(\omega)=(s,L)\}
    =
    \frac{M_n^\omega(s)}{u(s)}.
\]
The shape laws \((M_n^\omega)_{n\ge1}\) form a coherent family for the
leaf-grafting graph:
\[
    M_n^\omega(s)
    =
    \sum_{T:s\nearrow T}
    M_{n+1}^\omega(T)
    \frac{u(s)n(s;T)}{u(T)},
    \qquad s\in\T_n.
\]
Under the coupling of \zcref{lem:paintbox-consistency}, the bijections
of \zcref{prop:path-quotient-labelling-bijection} identify the nested
samples with an infinite path. Its law, denoted by \(\mu^\omega\), is
central and has marginals \((M_n^\omega)_{n\ge1}\).
\end{proposition}

\begin{proof}
We prove the first identity by induction on \(n\). For \(n=1\), both sides equal one.
Fix \(n\ge2\), \(s\in\T_n\), and \(L\in\QSL(s)\).
Index the root branches of \((s,L)\) by \(1,\ldots,r\) and write
their shapes, quotient labellings, and label sets as
\[
    (s_1,L_1;B_1),\ldots,(s_r,L_r;B_r),
    \qquad
    B_1\sqcup\cdots\sqcup B_r=\{1,\ldots,n-1\}.
\]
On the event \(\{\widehat T_n=(s,L)\}\), let \(\alpha(j)\) be the
common value of \(J_{m,\varnothing}\) for \(m\in B_j\).
No two branches use the same positive box, and the dust value
\(\alpha(j)=0\) is possible only if \(B_j\) is a singleton.
We call any map \(\alpha:\{1,\ldots,r\}\to\Nzero\) satisfying these
two conditions an admissible root assignment.
Summing over admissible assignments and using the independence of the
child samples gives
\begin{equation}\label{eq:paintbox-labelled-prob-recursion}
\begin{aligned}
    \Prob_\omega\{\widehat T_n=(s,L)\}
    =\sum_{\alpha\ \mathrm{admissible}}
       \prod_{j:\alpha(j)=0}p_0
       \prod_{j:\alpha(j)>0}
       \left[
         p_{\alpha(j)}^{|B_j|}
         \Prob_{\omega_{\alpha(j)}}
         \{\widehat T_{B_j}=(s_j,L_j)\}
       \right].
\end{aligned}
\end{equation}
Since \(|B_j|=|s_j|<n\), \zcref{lem:paintbox-consistency} and the
induction hypothesis identify each child probability in
\eqref{eq:paintbox-labelled-prob-recursion} with
\(M_{|s_j|}^{\omega_{\alpha(j)}}(s_j)/u(s_j)\).
Permuting the branches only permutes the admissible assignments in
\eqref{eq:paintbox-labelled-prob-recursion}. The sum therefore depends
only on the multiset of branch shapes, hence only on \(s\).
The \(u(s)\) quotient labellings of \(s\) share the total probability
\(M_n^\omega(s)\) equally.

By \zcref{lem:paintbox-consistency,prop:path-quotient-labelling-bijection},
the nested samples determine an infinite path. For a finite path
\(\gamma\) ending at \(s\in\T_n\), let \(L\in\QSL(s)\) be its
corresponding quotient labelling. The sampled path begins with
\(\gamma\) precisely when \(\widehat T_n=(s,L)\), so
\(\mu^\omega(C_\gamma)=M_n^\omega(s)/u(s)\).
Thus \(\mu^\omega\) is central, and its marginals satisfy the stated
coherence relation by \eqref{eq:coherent-family-definition}.
\end{proof}

\begin{definition}[Recursive-paintbox boundary]
\label{def:recursive-paintbox-boundary}
Two deterministic recursive paintboxes are \emph{sampling equivalent}, written
\(\omega\sim_{\mathrm{samp}}\omega'\), when
\(M_n^\omega=M_n^{\omega'}\) for every \(n\). By
\zcref{prop:paintbox-central}, this is equivalent to equality of their finite
quotient-labelled laws, or of their central measures.

Let \(\partial_{\mathrm{RP}}\) be the set of sampling equivalence classes,
and write \(M_n^\xi=M_n^\omega\) for \(\xi=[\omega]_{\mathrm{samp}}\).
Equip it with the topology induced by the injective map
\[
    \Psi:\partial_{\mathrm{RP}}\longrightarrow
         \prod_{n\ge1}\Delta(\T_n),\qquad
    \Psi(\xi)=(M_n^\xi)_{n\ge1},
\]
where \(\Delta(\T_n)\) is the finite-dimensional probability simplex on
\(\T_n\). We regard \(\Psi(\xi)\) as a point of \([0,1]^{\T}\) by writing its
coordinate at \(s\) as \(\Psi(\xi)_s=M_{|s|}^\xi(s)\).
Thus convergence in \(\partial_{\mathrm{RP}}\) is equivalent to
convergence in total variation of the sampling laws at each fixed size.
\end{definition}

\begin{remark}[Sampling equivalence and relabelling]
\label{rem:paintbox-equivalence-relabelling}
For a deterministic recursive paintbox, 
permuting children together with their descendant environments preserves
the sampling laws. Conversely, we will prove in
\zcref{cor:paintbox-sampling-equivalence} that, after deleting zero-mass
branches and their descendants, sampling-equivalent representatives differ
only by a relabelling that preserves ancestry and every relative branch mass.
\end{remark}

Let
\[
    \Delta_\downarrow=
       \{(p_i)_{i\ge1}:p_1\ge p_2\ge\cdots\ge0,\ \sum_i p_i\le1\},
\]
and equip \(\Delta_\downarrow\) with the topology of coordinatewise
convergence and \(\Delta_\downarrow^{\Ulam}\) with the product topology.
Each array in \(\Delta_\downarrow^{\Ulam}\) defines a recursive paintbox
with dust mass \(p_{w,0}=1-\sum_{i\ge1}p_{w,i}\) at each address \(w\).
Let \(q:\Delta_\downarrow^{\Ulam}\to\partial_{\mathrm{RP}}\) be the
projection that sends \(\omega\) to its sampling class
\([\omega]_{\mathrm{samp}}\).

\begin{proposition}[Continuity of finite sampling]
\label{prop:environment-sampling-continuity}
For each \(n\ge1\), the map
\(\Delta_\downarrow^{\Ulam}\to\Delta(\T_n)\), \(\omega\mapsto M_n^\omega\), is
continuous in total variation. Consequently, \(q\) is continuous and
\(\partial_{\mathrm{RP}}\) is compact metrizable.
\end{proposition}
\begin{proof}
We argue by induction on \(n\). For \(n=1\), the sampling law is constant.
Fix \(n\ge2\). To truncate the sample, retain the first \(L\) boxes at
the root and treat each label that chooses a later box as a separate
dust leaf. Use the same box choices in both samples. The two trees agree unless some box \(i>L\) receives at least two
of the \(n-1\) non-root labels. Writing \(M_{n,L}^\omega\) for the truncated shape law gives
\begin{equation}\label{eq:uniform-root-truncation-continuity}
    d_{\mathrm{TV}}(M_n^\omega,M_{n,L}^\omega)
       \le\binom{n-1}{2}\sum_{i>L}p_i^2
       \le\frac{\binom{n-1}{2}}{L+1}.
\end{equation}
The last inequality uses
\(p_{L+1}\le1/(L+1)\) and
\(\sum_{i>L}p_i^2\le p_{L+1}\sum_{i>L}p_i\).

In the truncated sample, fix an allocation of the \(n-1\) non-root
labels to boxes \(1,\ldots,L\) or dust. If \(m_i\) labels choose box
\(i\) and \(m_0\) choose dust,
the probability of this allocation is
\((1-\sum_{i=1}^L p_i)^{m_0}\prod_{i=1}^L p_i^{m_i}\).
Given this allocation, the dust labels form \(m_0\) leaves at the root.
The occupied boxes give independent subtrees, with shape law
\(M_{m_i}^{\omega_i}\) in box \(i\) by \zcref{lem:paintbox-consistency}.
Since \(m_i<n\) and \(\omega\mapsto\omega_i\) is continuous, the
induction hypothesis gives continuity of each child law. Summing over
the finitely many allocations gives continuity of
\(\omega\mapsto M_{n,L}^\omega\).
By \eqref{eq:uniform-root-truncation-continuity}, these maps converge
uniformly in total variation to \(\omega\mapsto M_n^\omega\) as
\(L\to\infty\), so this map is continuous as well.
Finally, \(M_n^{q(\omega)}=M_n^\omega\) for every \(n\), and the maps
\(\xi\mapsto M_n^\xi\) induce the sampling topology. Hence \(q\) is
continuous.

The inequalities \(p_i\ge p_{i+1}\ge0\) and
\(\sum_{i=1}^{L}p_i\le1\), for every \(i,L\ge1\), define
\(\Delta_\downarrow\) as a closed subset of \([0,1]^\N\).
Thus \(\Delta_\downarrow^{\Ulam}\) is compact metrizable.
Since \(q\) is onto and \(\partial_{\mathrm{RP}}\) is a subspace of a
countable product of finite simplices, the sampling boundary is compact
metrizable as well.
\end{proof}

\subsection{Root decomposition of a uniform quotient labelling}

To distinguish isomorphic root branches, we lift a uniform quotient
labelling to a uniform standard labelling.

\begin{lemma}[Root decomposition of a uniform quotient labelling]
\label{lem:root-decomposition-uniform}
Let \(N\ge2\) and \(t\in\T_N\). Choose a representative of
\(t\) and index its root subtrees by \(1,\ldots,r\), distinguishing
isomorphic branches. Let their shapes be \(t_1,\ldots,t_r\), with
\(|t_i|=N_i\) and \(\sum_{i=1}^r N_i=N-1\).

Let \([L]\) be uniform on \(\QSL(t)\). Conditional on \([L]\), choose
\(\widetilde L\) uniformly among the
\(|\Aut(t)|\) labellings in the orbit \([L]\subseteq
\SL(t)\). Then \(\widetilde L\) is uniform on \(\SL(t)\).

For \(1\le i\le r\), let \(\mathcal A_i\) be the set of labels assigned
by \(\widetilde L\) to the \(i\)th root subtree, and let
\(\mathsf L_i\in\QSL_{\mathcal A_i}(t_i)\) be the quotient class of the
restriction of \(\widetilde L\) to that subtree. Then:
\begin{enumerate}[label=(\roman*)]
    \item the ordered partition
    \((\mathcal A_1,\ldots,\mathcal A_r)\) is uniform among all ordered
    partitions of \(\{1,\ldots,N-1\}\) satisfying
    \(|\mathcal A_i|=N_i\). In particular, for every fixed
    \(R\subseteq\{1,\ldots,N-1\}\), the vector
    \((|\mathcal A_1\cap R|,\ldots,|\mathcal A_r\cap R|)\)
    has the multivariate hypergeometric law with cell sizes
    \(N_1,\ldots,N_r\) and sample size \(|R|\).

    \item conditional on
    \((\mathcal A_1,\ldots,\mathcal A_r)\), the random variables
    \(\mathsf L_1,\ldots,\mathsf L_r\) are independent, and
    \(\mathsf L_i\) is uniform on
    \(\QSL_{\mathcal A_i}(t_i)\).
\end{enumerate}

\end{lemma}

\begin{proof}
The action of \(\Aut(t)\) on \(\SL(t)\) is free, so all orbits have
cardinality \(|\Aut(t)|\). Choosing a uniform orbit and then a uniform
representative makes \(\widetilde L\) uniform
on \(\SL(t)\).

Fix an ordered partition \((A_1,\ldots,A_r)\).
Restriction to the branches gives a bijection between the standard
labellings of \(t\) with label set \(A_i\) on each branch \(t_i\)
and \(\prod_i \SL_{A_i}(t_i)\). Their number, \(\prod_i d(t_i)\), does not
depend on the label sets.
Thus the induced partition is uniform. For any fixed subset of labels,
the numbers assigned to the respective branches have the stated
multivariate hypergeometric law.

Conditional on these label sets, every tuple in
\(\prod_i\SL_{A_i}(t_i)\) is equally likely, so the branch labellings
are independent and uniform. Passing to quotient classes preserves
independence, and since \(\Aut(t_i)\) acts freely on the labellings of each
branch, the quotient classes are uniform as well.
\end{proof}

\subsection{Sampling and dust truncation}

We record the truncation bound for both independent sampling and
sampling without replacement, then use it to pass to limits in the
root decomposition.

\begin{lemma}[Dust truncation]
\label{lem:sampling-truncation-bounds}
Suppose \(r\ge0\) labels independently choose boxes with masses
\((q_i)_{i\ge1}\), where \(\sum_iq_i\le1\), or dust with the remaining
probability. For \(L\ge0\), replace every label assigned to a box
\(i>L\) by a separate singleton. The probability that this changes the
induced partition is at most \(\binom r2\sum_{i>L}q_i^2\).
The same bound holds when the \(r\) labels are drawn uniformly without
replacement from a population of size \(M\ge\max\{r,1\}\) containing
disjoint cells of sizes \(Mq_i\). A label that draws an element outside
these cells is treated as dust.
\end{lemma}
\begin{proof}
Truncation changes the partition precisely when two labels fall in the
same box with index greater than \(L\). For independent choices, each pair has probability
\(\sum_{i>L}q_i^2\) of doing so. For sampling without replacement,
write \(M_i=Mq_i\) for the cell sizes. When \(r\ge2\), this probability
is
\[
    \sum_{i>L}\frac{M_i(M_i-1)}{M(M-1)}
       \le\sum_{i>L}\left(\frac{M_i}{M}\right)^2.
\]
A union bound proves the claim. For \(r\le1\), there are no pairs of labels, so the probability that
truncation changes the partition is zero.
\end{proof}

The next lemma turns convergence of the root masses and the child
sampling laws into convergence of the whole sample. We will apply it
to the nested samples in \zcref{thm:nested-paintbox-limit}.
\begin{lemma}[Root sampling and child limits]
\label{lem:root-sampling-with-child-limits}
Fix \(n\ge2\), a sequence of finite trees \((t_j)\) with
\(|t_j|\to\infty\), and a recursive paintbox \(\omega\) with root masses
\(p_i\) and child environments \(\omega_i\). For each \(i\ge1\) with
\(p_i>0\), let \(a_{j,i}\) be a root branch of \(t_j\), defined for all
sufficiently large \(j\), and assume that
\[
    \frac{|a_{j,i}|}{|t_j|-1}\longrightarrow p_i,
    \qquad
    K_s(a_{j,i})\longrightarrow M_{|s|}^{\omega_i}(s)
       \quad\text{for every }s\in\T\text{ with }|s|<n.
\]
For \(L\ge1\), put \(I_L=\{1\le i\le L:p_i>0\}\). Assume that for
each fixed \(L\) the branches \(a_{j,i}\), \(i\in I_L\), have distinct
roots for all sufficiently large \(j\). Let \(\mathcal R_{j,L}\) consist
of the remaining root branches of \(t_j\). If
\begin{equation}\label{eq:root-sampling-residual-collision}
    \lim_{L\to\infty}\limsup_{j\to\infty}
       \sum_{a\in\mathcal R_{j,L}}
          \left(\frac{|a|}{|t_j|-1}\right)^2=0,
\end{equation}
then \(K_s(t_j)\to M_n^\omega(s)\) for every \(s\in\T_n\).
\end{lemma}
\begin{proof}
Fix \(L\). For sufficiently large \(j\), choose a uniform standard
labelling of a representative of \(t_j\). For each
\(i\in I_L\), let \(m_i\) count the labels in \(\{1,\ldots,n-1\}\)
that lie in branch \(a_{j,i}\) in the finite-tree sample or are
assigned to box \(i\) in the paintbox sample. In either case, write
\(m=(m_i)_{i\in I_L}\) and \(m_0=n-1-\sum_{i\in I_L}m_i\).
We truncate either sample by attaching the \(m_0\) remaining labels
to the root as separate leaves and retaining the sampled subtree
for each \(i\in I_L\) with \(m_i>0\).

By part~(i) of \zcref{lem:root-decomposition-uniform}, the count vector
\((m_0,m)\) in the finite-tree sample has a multivariate hypergeometric law. Since \(|a_{j,i}|/(|t_j|-1)\to p_i\) for every \(i\in I_L\), the proportion of non-root vertices in the remaining branches tends to \(1-\sum_{i\in I_L}p_i\). Thus, for each fixed \(m\),
\[
    \frac{\binom{\sum_{a\in\mathcal R_{j,L}}|a|}{m_0}
          \prod_{i\in I_L}\binom{|a_{j,i}|}{m_i}}
         {\binom{|t_j|-1}{n-1}}
    \longrightarrow
    \frac{(n-1)!}{m_0!}
    \left(1-\sum_{i\in I_L}p_i\right)^{m_0}
    \prod_{i\in I_L}\frac{p_i^{m_i}}{m_i!}.
\]
In the paintbox sample, the right-hand side is the probability that each box \(i\in I_L\) receives exactly \(m_i\) labels.

Given the label sets of the root branches of \(t_j\), their quotient
labellings are independent and uniform by part~(ii) of
\zcref{lem:root-decomposition-uniform}. For each \(i\in I_L\) with \(m_i>0\), the sampled subtree is induced
by the \(m_i\) smallest labels in \(a_{j,i}\). By \zcref{prop:sampling-kernel}, the conditional probability of prescribed
shapes \(s_i\in\T_{m_i}\) in these branches is
\(\prod_{i\in I_L:\,m_i>0}K_{s_i}(a_{j,i})\).
In the paintbox sample, the corresponding conditional probability is
\(\prod_{i\in I_L:\,m_i>0}M_{m_i}^{\omega_i}(s_i)\), since the child
samples are independent given the allocation. In both models, this
probability depends on the allocation only through \(m\).
Since \(m_i<n\) for every occupied box, the first product converges to
the second by hypothesis. Multiplying by the respective probabilities
of the counts gives convergence of the joint probabilities of the
counts and prescribed shapes.

For \(s\in\T_n\), sum these joint probabilities over the choices of
\(m\) and \((s_i)_{i\in I_L:\,m_i>0}\) for which attaching the \(s_i\)
and \(m_0\) leaves to a new root gives \(s\). There are only finitely many such choices, so the probability that the
truncated finite-tree sample has shape \(s\) converges to the corresponding
probability for the truncated paintbox sample. Since \(\T_n\) is finite,
the truncated laws converge in total variation for each fixed \(L\).

Truncation can change the sampled shape only if two labels lie in
the same branch of \(\mathcal R_{j,L}\) in the finite-tree sample
or choose the same box \(i>L\) in the paintbox sample. Applying \zcref{lem:sampling-truncation-bounds}
to both samples and using the triangle inequality gives
\[
    \limsup_{j\to\infty}
      d_{\mathrm{TV}}\bigl((K_s(t_j))_{s\in\T_n},M_n^\omega\bigr)
    \le \binom{n-1}{2}
      \left(
        \limsup_{j\to\infty}
        \sum_{a\in\mathcal R_{j,L}}
          \left(\frac{|a|}{|t_j|-1}\right)^2
        +\sum_{i>L}p_i^2
      \right).
\]
As \(L\to\infty\), the first term in parentheses tends to zero by
\eqref{eq:root-sampling-residual-collision}, and the second tends to
zero since \(\sum_{i\ge1}p_i^2\le1\).
\end{proof}

\section{The Martin boundary and its central measures}
\label{sec:recursive-paintbox-boundary}
We now prove Theorem A and Corollary B.

\subsection{Finite trees and their recursive paintboxes}

Order the children at every vertex of a finite tree \(t\) by
nonincreasing subtree size, breaking ties according to a fixed total order on rooted-tree shapes.
This defines a Ulam representation of \(t\). The root has the empty word as its address, and the \(i\)th child of a
vertex with address \(w\) has address \(wi\). We write \(t_w\)
for the subtree at address \(w\), with size zero at missing addresses.
Define \(\omega(t)\in\Delta_\downarrow^{\Ulam}\) by the local masses
\[
    p_{w,i}^{(t)}=
    \begin{cases}
        |t_{wi}|/(|t_w|-1),&|t_w|\ge2,\\
        0,&|t_w|\le1,
    \end{cases}
    \qquad i\ge1.
\]
Thus the local dust mass is \(1\) at every leaf and every missing address. At the root, the
environment of the \(i\)th child is \(\omega(t_i)\). Permuting isomorphic
children leaves the array unchanged.

\begin{proposition}[Uniform approximation by a recursive paintbox]
\label{prop:finite-tree-paintbox-approximation}
For integers \(2\le n\le N\),
\begin{equation}\label{eq:finite-tree-paintbox-approximation}
    \sup_{t\in\T_N}d_{\mathrm{TV}}\!\left(
        (K_s(t))_{s\in\T_n},M_n^{\omega(t)}
    \right)
    =1-\prod_{j=0}^{n-2}\left(1-\frac{j}{N-1}\right)
    \le\frac{\binom{n-1}{2}}{N-1}.
\end{equation}
The supremum is attained by the star \(S_N\). Consequently, for any
integer sequence \(2\le n=n(N)\le N\), the approximation error tends
to zero uniformly over \(t\in\T_N\) if and only if \(n=o(\sqrt N)\).
\end{proposition}
\begin{proof}
Fix \(t\in\T_N\) and the representative used to define
\(\omega(t)\).
A uniformly chosen standard labelling projects to a uniform quotient
standard labelling, since all automorphism orbits have the same size.
Its restriction to the first \(n\) labels therefore has shape law
\((K_s(t))_{s\in\T_n}\). We couple this labelling with a paintbox sample from \(\omega(t)\), one label at a time.

Fix the positions of the first \(k\) labels, where \(1\le k<N\).
The set \(A\subseteq V(t)\) of vertices they occupy is ancestor-closed.
Let \(F(A)\) consist of
the vertices outside \(A\) whose parents belong to \(A\)
(see \zcref{fig:finite-tree-paintbox}). Their
descendant subtrees partition \(V(t)\setminus A\), so
\(\sum_{v\in F(A)}h_v(t)=N-k\). By the hook formula for the remaining forest, the number of standard
labellings of \(t\) in which \(0,\ldots,k-1\) occupy their prescribed
vertices is \((N-k)!/\prod_{x\notin A}h_x(t)\).
For a fixed \(v\in F(A)\), the number of these labellings that assign
label \(k\) to \(v\) is \((N-k-1)!/\prod_{x\notin A\cup\{v\}}h_x(t)\).
Thus \(v\) receives label \(k\) with conditional probability
\(h_v(t)/(N-k)\).

Under paintbox sampling, a label that reaches an occupied non-leaf
vertex \(w\) moves to its \(i\)th child \(wi\) with probability
\(p_{w,i}^{(t)}=|t_{wi}|/(|t_w|-1)\). 
These choices are independent across labels and vertices. Each label stops at the first unoccupied vertex it reaches. If it reaches
an occupied leaf of \(t\), it creates a dust child there. Suppose that each of the first \(k\) labels occupies the same vertex in both samples.
For \(v\in F(A)\), let \(v_0,\ldots,v_d\) be the path from the root to \(v\).
Under paintbox sampling, the conditional probability that \(v\) receives
label \(k\) is
\begin{equation}\label{eq:finite-tree-paintbox-frontier}
    \prod_{j=1}^{d}\frac{h_{v_j}(t)}{h_{v_{j-1}}(t)-1}
    = \frac{h_v(t)}{N-1}
      \prod_{j=1}^{d-1}\frac{h_{v_j}(t)}{h_{v_j}(t)-1}
    \ge \frac{h_v(t)}{N-1}.
\end{equation}
Thus, under both uniform standard labelling and paintbox sampling, each \(v\in F(A)\) receives label \(k\) with conditional probability at least \(h_v(t)/(N-1)\). Under a maximal coupling of these conditional laws, label \(k\) occupies the same vertex in both samples with probability at least \(\sum_{v\in F(A)}\frac{h_v(t)}{N-1}=\frac{N-k}{N-1}\).

Use this coupling until the first disagreement and then complete the two samples independently. With probability at least
\(\prod_{k=1}^{n-1}\frac{N-k}{N-1}\), the two samples agree on the
positions of all \(n\) labels. Their shapes then agree, so the
coupling inequality gives
\[
    d_{\mathrm{TV}}\!\left((K_s(t))_{s\in\T_n},M_n^{\omega(t)}\right)
    \le 1-\prod_{j=0}^{n-2}\left(1-\frac{j}{N-1}\right)
    \le \frac{\binom{n-1}{2}}{N-1},
\]
where the last step uses \(1-\prod_j(1-a_j)\le\sum_j a_j\) for
\(0\le a_j\le1\).

For \(t=S_N\), the finite restriction has shape \(S_n\) with probability
one. The environment \(\omega(S_N)\) has \(N-1\) root boxes of mass
\(1/(N-1)\), each with a pure-dust environment. Its sample has shape
\(S_n\) precisely when the \(n-1\) non-root labels choose distinct boxes.
Thus \(d_{\mathrm{TV}}(\delta_{S_n},M_n^{\omega(S_N)})
=1-\prod_{j=0}^{n-2}(1-j/(N-1))\), so \(S_N\) attains the supremum.
Since \(1-x\le\mathrm e^{-x}\), this distance is at least
\(1-\mathrm e^{-\binom{n-1}{2}/(N-1)}\). Together with the upper
bound \(\binom{n-1}{2}/(N-1)\), this shows that the uniform error
tends to zero if and only if \(n=o(\sqrt N)\).
\end{proof}

\begin{figure}[ht]
\centering
\begin{tikzpicture}[x=0.95cm,y=0.9cm,
  hook/.style={font=\scriptsize,text=black!70,inner sep=1.5pt},
  frontier/.style={circle,draw=figblue,line width=0.7pt,
    inner sep=0pt,minimum size=9pt}]
  \coordinate (r) at (1.75,0);
  \coordinate (a) at (0.5,-1);
  \coordinate (b) at (2,-1);
  \coordinate (c) at (3,-1);
  \coordinate (a1) at (0,-2);
  \coordinate (a2) at (1,-2);
  \coordinate (b1) at (2,-2);
  \coordinate (a11) at (0,-3);
  \draw[fig edge]
    (r)--(a) (r)--(b) (r)--(c)
    (a)--(a1) (a)--(a2) (a1)--(a11) (b)--(b1);
  \foreach \v/\l in {r/0,a/1,b/2}
    \node[fig label,minimum size=12pt,fill=gray!16]
      at (\v) {\(\l\)};
  \foreach \v in {c,a2,b1,a11}
    \node[fig vertex] at (\v) {};
  \foreach \v in {c,a2,b1}
    \node[frontier] at (\v) {};
  \node[frontier,fill=figblue!22] at (a1) {};
  \node[fig vertex,fill=figblue] at (a1) {};
  \node[hook,left=4pt] at (a1) {\(h_v(t)=2\)};
  \node[hook,right=4pt] at (a2) {\(1\)};
  \node[hook,right=4pt] at (b1) {\(1\)};
  \node[hook,right=4pt] at (c) {\(1\)};
  \node[fig panel] at (1.5,-4.15) {(a) \(t\in\T_8\)};

  \begin{scope}[xshift=6cm]
    \node[fig env] (R) at (1.75,0) {};
    \node[fig env] (A) at (0.5,-1) {};
    \node[fig env] (B) at (2,-1) {};
    \node[fig env] (C) at (3,-1) {};
    \node[fig env,draw=figblue,fill=figblue!22,minimum size=9pt]
      (A1) at (0,-2) {};
    \node[fig env] (A2) at (1,-2) {};
    \node[fig env] (B1) at (2,-2) {};
    \node[fig env] (A11) at (0,-3) {};
    \draw[fig edge,figblue]
      (R) -- node[fig mass,above left] {\(\tfrac47\)} (A);
    \draw[fig edge]
      (R) -- node[fig mass,left,pos=0.55] {\(\tfrac27\)} (B);
    \draw[fig edge]
      (R) -- node[fig mass,above right] {\(\tfrac17\)} (C);
    \draw[fig edge,figblue]
      (A) -- node[fig mass,left] {\(\tfrac23\)} (A1);
    \draw[fig edge]
      (A) -- node[fig mass,right] {\(\tfrac13\)} (A2);
    \draw[fig edge]
      (A1) -- node[fig mass,left] {\(1\)} (A11)
      (B) -- node[fig mass,left] {\(1\)} (B1);
    \node[fig note,text=figblue,left=4pt] at (A1.west) {\(v\)};
    \foreach \v in {C,A2,B1,A11}{
      \draw[fig dust] (\v) -- ++(0,-0.42);
      \pic[scale=1.15] at ($(\v)+(0,-0.55)$) {fig dustcloud};
    }
    \node[fig panel] at (1.5,-4.15) {(b) \(\omega(t)\)};
  \end{scope}
\end{tikzpicture}
\caption{(a) A tree \(t\in\T_8\). The vertices labelled \(0,1,2\) form
the occupied set \(A\), and the circled unlabelled vertices form  \(F(A)\).
Numbers beside the vertices in \(F(A)\) are hook lengths.
(b) The corresponding paintbox \(\omega(t)\), with relative child masses
on solid edges and pure dust at the leaves. Given the positions of
\(0,1,2\) shown in (a), the next label occupies the highlighted vertex
\(v\) with probability \(2/5\) under uniform labelling and
\((4/7)(2/3)=8/21\) under paintbox sampling.
Both probabilities are at least \(h_v(t)/7=2/7\).}
\label{fig:finite-tree-paintbox}
\end{figure}

\begin{corollary}[Paintbox representations of subsequential Martin limits]
\label{cor:subseq-to-paintbox}
Let \((t_j)\) be finite rooted trees with \(|t_j|\to\infty\). There exists a deterministic recursive paintbox \(\omega\) such that along
some subsequence \((t_{j_\ell})\) we have, for every rooted tree \(s\),
\[
    K_s(t_{j_\ell})\longrightarrow M_{|s|}^\omega(s).
\]
\end{corollary}
\begin{proof}
Compactness of \(\Delta_\downarrow^{\Ulam}\) gives a subsequence along which
\(\omega(t_{j_\ell})\to\omega\). By
\zcref{prop:environment-sampling-continuity},
\(M_n^{\omega(t_{j_\ell})}\to M_n^\omega\) in total variation for
every fixed \(n\). Since \(|t_{j_\ell}|\to\infty\), the bound in
\eqref{eq:finite-tree-paintbox-approximation} tends to zero for fixed \(n\).
Thus \((K_s(t_{j_\ell}))_{s\in\T_n}\) converges to \(M_n^\omega\)
in total variation.
\end{proof}

\subsection{Nested samples and extremality}

Fix a deterministic recursive paintbox \(\omega\) and couple its samples
\(T_N(\omega)\) as in \zcref{lem:paintbox-consistency}.
We now show that, along this sequence, 
the Martin coordinates converge almost surely
to the sampling probabilities of \(\omega\).

\begin{theorem}[Nested samples and extremality]
\label{thm:nested-paintbox-limit}
For every deterministic recursive paintbox \(\omega\), its nested sample
\(X_N=T_N(\omega)\) satisfies, simultaneously for all \(s\in\T\),
\begin{equation}\label{eq:nested-sample-martin-limit}
    K_s(X_N)\longrightarrow M_{|s|}^{\omega}(s)
       \quad\text{almost surely}.
\end{equation}
Consequently \(\mu^\omega\) is extremal. There is also a deterministic
sequence \(t_N\in\T_N\) with \(t_N\nearrow t_{N+1}\) for every \(N\)
and \(K_s(t_N)\to M_{|s|}^{\omega}(s)\) for every \(s\in\T\).
\end{theorem}
\begin{proof}
We prove by induction on \(n\) that, for each deterministic environment,
the asserted limits hold almost surely for all \(s\in\T_n\).
The case \(n=1\) is immediate.
Fix \(n\ge2\). Write \(p_i\) and \(\omega_i\) for the mass and
environment of root box \(i\ge1\). Using the choices in
\eqref{eq:uniform-routing-realization}, set
\(B_i=\{m\ge1:J_{m,\varnothing}=i\}\) and
\(B_i(N)=B_i\cap\{1,\ldots,N-1\}\).
Almost surely, the strong law gives \(|B_i(N)|/(N-1)\to p_i\) for
every \(i\ge1\). In particular, \(|B_i(N)|\to\infty\) whenever \(p_i>0\).

Fix \(L\ge1\) and consider the root branches of \(X_N\).
Let \(\mathcal R_{N,L}\) consist of the subtrees in occupied boxes
\(i>L\) and all dust leaves at the root.
Each occupied box \(i>L\) contains \(|B_i(N)|\le N-1\) vertices,
and there are at most \(N-1\) dust leaves. Hence
\[
    \sum_{a\in\mathcal R_{N,L}}\left(\frac{|a|}{N-1}\right)^2
    \le \frac{\sum_{i>L}|B_i(N)|}{N-1}+\frac{1}{N-1}.
\]
The first term on the right is the proportion of labels \(1\le m<N\)
for which \(J_{m,\varnothing}>L\). By the strong law, this proportion
converges almost surely to \(\sum_{i>L}p_i\), simultaneously for all
\(L\ge1\). Since \(\sum_{i>L}p_i\to0\) as \(L\to\infty\),
\eqref{eq:root-sampling-residual-collision} follows.

We now condition on the label sets \((B_i)_{i\ge1}\) at the root. For \(p_i>0\), let \(a_{N,i}\) be the
subtree in box \(i\) once the box is occupied. The box choices \(J_{m,iw}\) in branch \(i\) are independent of the
allocation at the root. Thus, inserting the labels of \(B_i\) in
increasing order gives a nested sample from \(\omega_i\). 
Since \(|a_{N,i}|=|B_i(N)|\to\infty\),
the induction hypothesis gives
\[
    K_s(a_{N,i})\longrightarrow M_{|s|}^{\omega_i}(s),
    \qquad |s|<n,
\]
with conditional probability one. There are only countably many positive
boxes, so the limits hold for all of them on an event of probability one.
Distinct occupied boxes give distinct root branches \(a_{N,i}\).
Their relative sizes converge to \(p_i\), and the remaining branches
satisfy the residual bound above. Thus
\zcref{lem:root-sampling-with-child-limits} gives
\(K_s(X_N)\to M_n^\omega(s)\) for every \(s\in\T_n\).
This completes the induction, and a countable intersection gives
\eqref{eq:nested-sample-martin-limit} for all shapes.

By \zcref{prop:paintbox-central}, the nested samples define
an infinite path with law \(\mu^\omega\).
Since \(M_{|s|}^{\omega}(s)=\mu^\omega(X_{|s|}=s)\),
\zcref{prop:central-kernels-tail} gives extremality.
Finally, fix a sample path on which \eqref{eq:nested-sample-martin-limit}
holds for every \(s\in\T\), and let \(t_N\) be its tree shape at level
\(N\). Each step adds a leaf, so \(t_N\nearrow t_{N+1}\) for every \(N\).
\end{proof}

\subsection{Boundary classification and minimality}

\begin{theorem}[Rooted-tree Martin boundary classification]
\label{thm:complete-boundary}
The sampling map \(\Psi\) of \zcref{def:recursive-paintbox-boundary}
is a homeomorphism
\[
    \Psi:\partial_{\mathrm{RP}}
       \xrightarrow{\ \cong\ }
       \partial\Omega_M=\partial_{\min}\Omega_M,
\]
satisfying \(K_s(\Psi(\xi))=M_{|s|}^{\xi}(s)\) for every
\(\xi\in\partial_{\mathrm{RP}}\) and \(s\in\T\).

For every \(\xi\in\partial_{\mathrm{RP}}\), there is a sequence
\(t_N\in\T_N\) with \(t_N\nearrow t_{N+1}\) such that
\(K_s(t_N)\to M_{|s|}^{\xi}(s)\) for every \(s\in\T\).
\end{theorem}

\begin{proof}
Let \(\zeta\in\partial\Omega_M\). By
\zcref{lem:finite-vertices-isolated}, choose \(t_j\) with \(|t_j|\to\infty\)
and \(\iota_M(t_j)\to\zeta\). By \zcref{cor:subseq-to-paintbox}, there are
a subsequence \((t_{j_\ell})\) and a recursive paintbox \(\omega\) such that
\[
    K_s(\zeta)=\lim_\ell K_s(t_{j_\ell})
       =M_{|s|}^\omega(s),\qquad s\in\T.
\]
Conversely, fix \(\xi=[\omega]_{\mathrm{samp}}\in\partial_{\mathrm{RP}}\).
By \zcref{thm:nested-paintbox-limit}, there are trees \(t_N\in\T_N\)
with \(t_N\nearrow t_{N+1}\) such that
\(\iota_M(t_N)\to\Psi(\xi)\) coordinatewise.
Thus \(\Psi(\xi)\in\Omega_M\). Since \(|t_N|=N\to\infty\) and every
finite tree is isolated in \(\Omega_M\) by \zcref{lem:finite-vertices-isolated},
the limit lies in \(\partial\Omega_M\).

Thus \(\Psi(\partial_{\mathrm{RP}})=\partial\Omega_M\). 
By \zcref{def:recursive-paintbox-boundary}, the sampling topology is
induced by the injective map \(\Psi\), so \(\Psi\) is a homeomorphism
onto \(\partial\Omega_M\).

Finally, write \(\zeta=\Psi([\omega]_{\mathrm{samp}})\).
Then \(K_s(\zeta)=M_{|s|}^\omega(s)\) for every \(s\in\T\), so the
central measure associated with \(\zeta\) is \(\mu^\omega\).
This measure is extremal by \zcref{thm:nested-paintbox-limit}, so
\(\zeta\) is minimal by
\zcref{prop:extremal-laws-minimal-boundary}.
\end{proof}

From now on we identify
\(\partial_{\mathrm{RP}}\) with \(\partial\Omega_M\) through \(\Psi\),
suppress \(\Psi\) from the notation, and write \(K_s(\xi)=M_{|s|}^{\xi}(s)\) for
\(\xi\in\partial_{\mathrm{RP}}\).

For \(\xi=[\omega]_{\mathrm{samp}}\), set \(\mu^\xi=\mu^\omega\).
Equality of sampling laws gives equality on path cylinders, so this measure is
independent of the representative. For \(s\in\T_n\) and a path
\(\gamma\) ending at \(s\),
\begin{equation}\label{eq:boundary-indexed-central-measure-definition}
    M_n^\xi(s)=K_s(\xi),\qquad
    \mu^\xi(C_\gamma)=\frac{K_s(\xi)}{u(s)}.
\end{equation}

\begin{remark}[Transition probabilities under \(\mu^\xi\)]
Let \(P\) be the transition kernel of the uniform recursive-tree chain
in \zcref{ex:uniform-recursive-chain}. For \(s\in\T_n\), set
\(h_\xi(s)=M_n^\xi(s)/M_n^{\mathrm{UR}}(s)=(n-1)!\,K_s(\xi)/u(s)\).
This function is nonnegative and \(P\)-harmonic by
\eqref{eq:boundary-harmonicity}.

For \(s\nearrow T\) with \(K_s(\xi)>0\), fix a path \(\gamma\) ending
at \(s\). Summing \(\mu^\xi(C_{\gamma e})/\mu^\xi(C_\gamma)\) over the
\(n(s;T)\) edges \(e\) from \(s\) to \(T\) and using
\eqref{eq:boundary-indexed-central-measure-definition} gives
\begin{equation}\label{eq:boundary-doob-transform}
    P^\xi(s,T)
    =\frac{n(s;T)u(s)K_T(\xi)}{u(T)K_s(\xi)}
    =P(s,T)\frac{h_\xi(T)}{h_\xi(s)}.
\end{equation}
Thus, under \(\mu^\xi\), the shape process \((X_N)_{N\ge1}\) is the
Doob \(h_\xi\)-transform of the uniform recursive-tree chain.
\end{remark}

\subsection{The canonical boundary variable and extremal decomposition}
We construct a Borel map \(\Xi\) on path space such that
\(X_N\to\Xi\) almost surely under every central measure. Its law under
\(\mu\) gives the mixing measure in the extremal decomposition of \(\mu\).
The completed tail fields of the shape process and the full path both
coincide with the completion of \(\sigma(\Xi)\). For ergodic decomposition by conditioning on the invariant \(\sigma\)-field,
see \cite[Theorem~A1.4]{kallenberg2005probabilistic}.

The map \(\xi\mapsto\mu^\xi\) is weakly continuous, since
\eqref{eq:boundary-indexed-central-measure-definition} gives continuous
probabilities on the clopen path cylinders. A monotone-class argument
shows that \(\xi\mapsto\mu^\xi(A)\) is Borel for every Borel path event
\(A\), so this family of measures forms a Borel probability kernel.

\begin{lemma}[A common boundary map]
\label{lem:raw-tail-borel-martin-limit}
There is a Borel map \(\Xi:\mathcal X_\Gamma\to\partial_{\mathrm{RP}}\)
measurable with respect to the uncompleted tail \(\sigma\)-field
\(\bigcap_{N\ge1}\sigma(X_m:m\ge N)\) of the shape process.
Under every central measure \(\mu\), almost surely we have
\begin{equation}\label{eq:raw-tail-martin-limit-coordinates}
    K_s(\Xi)=\lim_{N\to\infty}K_s(X_N)
       =\Prob_\mu(X_{|s|}=s\mid\mathscr G_\infty),
       \qquad s\in\T.
\end{equation}
Here
\(\mathscr G_\infty\) is the tail \(\sigma\)-field of the full path
(see \eqref{eq:raw-future-field-definition}). For every \(\xi\in\partial_{\mathrm{RP}}\),
\begin{equation}\label{eq:canonical-kernel-fibre-support}
    \mu^\xi\{\Xi=\xi\}=1.
\end{equation}
\end{lemma}
\begin{proof}
Let \(E\) be the set of paths along which \(K_s(X_N)\) converges for
every \(s\in\T\). For \(x\in E\), set
\(\Xi(x)=\lim_{N\to\infty}\iota_M(X_N(x))\) in \(\Omega_M\).
All level sums of this limit are one by \eqref{eq:finite-kernel-identities},
so it belongs to \(\partial\Omega_M\) by \zcref{lem:finite-vertices-isolated}.
On \(E^c\), set \(\Xi\) equal to a fixed boundary point. Each coordinate
limsup and liminf is measurable with respect to every
\(\sigma(X_m:m\ge N)\). Since \(\T\) is countable, \(E\) belongs to the uncompleted shape tail
\(\sigma\)-field. For each \(s\in\T\), \(K_s(\Xi)\) is tail-measurable because it equals
\(\limsup_{N\to\infty}K_s(X_N)\) on \(E\) and is constant on \(E^c\).
Since the coordinate maps \(K_s\) generate the Borel \(\sigma\)-field
on \(\partial_{\mathrm{RP}}\), \(\Xi\) is tail-measurable as well.
For every central \(\mu\), \zcref{prop:central-kernels-tail} gives
\(\mu(E)=1\) and \eqref{eq:raw-tail-martin-limit-coordinates}.
Under \(\mu^\xi\), \zcref{thm:nested-paintbox-limit} gives
\(K_s(X_N)\to K_s(\xi)\) almost surely for every \(s\). Since the
coordinates determine the boundary point, \(\Xi=\xi\) almost surely.
\end{proof}

\begin{theorem}[Canonical decomposition and tail fields]
\label{thm:universal-boundary-variable}
For a central measure \(\mu\), set \(\Theta_\mu=\mathcal L_\mu(\Xi)\)
and \(\Tcen=\overline{\mathscr G_\infty}^{\,\mu}\). Then
\begin{enumerate}[label=(\roman*),beginpenalty=10000]
\item The sequence \((X_N)\) converges almost surely to \(\Xi\)
in \(\Omega_M\), and
\(K_s(\Xi)=\Prob_\mu(X_{|s|}=s\mid\Tcen)\) for every \(s\).
\item The kernel \(\xi\mapsto\mu^\xi\) is a regular conditional law
of the full path given \(\Xi\).
\item The completed tail \(\sigma\)-fields satisfy
\[
    \Tcen
       =\overline{\bigcap_{N\ge1}\sigma(X_m:m\ge N)}^{\,\mu}
       =\overline{\sigma(\Xi)}^{\,\mu}.
\]
\item The representation
\begin{equation}\label{eq:canonical-boundary-disintegration-proof}
    \mu=\int_{\partial_{\mathrm{RP}}}\mu^\xi\,\Theta_\mu(d\xi)
\end{equation}
is unique among probability mixtures of the measures \(\mu^\xi\).
\end{enumerate}
In particular, \(\mu\) is extremal if and only if \(\Xi\) is almost
surely constant.
\end{theorem}
\begin{proof}
Part (i) follows from \zcref{lem:raw-tail-borel-martin-limit}.
For each finite path \(\gamma\) ending at \(s\), \eqref{eq:path-martin-general}
and reverse-martingale convergence give
\begin{equation}\label{eq:tail-conditional-kernel}
    \E_\mu[\ind_{C_\gamma}\mid\mathscr G_\infty]
       =\lim_{N\to\infty}\frac{K_s(X_N)}{u(s)}
       =\frac{K_s(\Xi)}{u(s)}
       =\mu^{\Xi}(C_\gamma).
\end{equation}
Since \(\xi\mapsto\mu^\xi\) is a Borel kernel, the monotone-class theorem
gives \(\Prob_\mu(A\mid\mathscr G_\infty)=\mu^{\Xi}(A)\) almost
surely for each Borel set \(A\subseteq\mathcal X_\Gamma\). The right-hand side is
\(\sigma(\Xi)\)-measurable. For each \(N\), the shape \(X_N\) and the subsequent edges determine
all \(X_m\) for \(m\ge N\), so \(\sigma(X_m:m\ge N)\subseteq\mathscr G_N\).
Together with \zcref{lem:raw-tail-borel-martin-limit}, this gives
\(\sigma(\Xi)\subseteq\bigcap_{N\ge1}\sigma(X_m:m\ge N)
\subseteq\mathscr G_\infty\).
Conditioning on \(\Xi\) now gives \(\Prob_\mu(A\mid\Xi)=\mu^\Xi(A)\)
almost surely, proving (ii).

For \(A\in\mathscr G_\infty\), we have
\(\mu^\Xi(A)=\Prob_\mu(A\mid\mathscr G_\infty)=\ind_A\) almost surely.
Thus \(A\) differs from the event \(\{\mu^\Xi(A)=1\}\in\sigma(\Xi)\)
by a \(\mu\)-null set. The inclusions \(\sigma(\Xi)
\subseteq\mathscr G_\infty\) in the preceding paragraph
therefore become equalities after \(\mu\)-completion, proving (iii).

Taking expectations in \(\Prob_\mu(A\mid\Xi)=\mu^\Xi(A)\) gives
\(\mu(A)=\int\mu^\xi(A)\,\Theta_\mu(d\xi)\) for each Borel \(A\),
which proves \eqref{eq:canonical-boundary-disintegration-proof}.

For uniqueness, suppose \(\mu=\int\mu^\xi\,\rho(d\xi)\).
By \eqref{eq:canonical-kernel-fibre-support},
\(\Xi_{\#}\mu^\xi=\delta_\xi\) for every \(\xi\). Taking the image
under \(\Xi\) gives \(\Theta_\mu=\Xi_{\#}\mu=\int\delta_\xi\,\rho(d\xi)=\rho\).
By part~(ii) of \zcref{prop:central-kernels-tail}, \(\mu\) is extremal
if and only if \(\mathscr G_\infty\) is \(\mu\)-trivial. By (iii), this
holds if and only if \(\sigma(\Xi)\) is \(\mu\)-trivial, that is,
if and only if \(\Xi\) is almost surely constant.
\end{proof}

Recall that the central probability measures on the path space
\(\mathcal X_\Gamma\) form a compact convex set \(\mathscr C\)
in the topology of weak convergence.
\begin{corollary}[The simplex of central measures]
\label{cor:bauer-simplex}
The barycentre map
\[
    \mathcal B:\mathcal P(\partial_{\mathrm{RP}})\longrightarrow\mathscr C,
    \qquad \rho\longmapsto\int\mu^\xi\,\rho(d\xi),
\]
is an affine homeomorphism. Its restriction to point masses identifies
\(\partial_{\mathrm{RP}}\) homeomorphically with
\(\operatorname{Ext}(\mathscr C)=\{\mu^\xi:\xi\in\partial_{\mathrm{RP}}\}\).
In particular, \(\mathscr C\) is a Bauer simplex.
\end{corollary}
\begin{proof}
The preceding theorem gives bijectivity. For every finite path
\(\gamma\) ending at \(s\),
\[
    \mathcal B(\rho)(C_\gamma)
       =\int\frac{K_s(\xi)}{u(s)}\,\rho(d\xi).
\]
The integrand is continuous, so cylinder probabilities converge whenever
the mixing measures converge weakly. Thus \(\mathcal B\) is a continuous bijection from the compact space
\(\mathcal P(\partial_{\mathrm{RP}})\) onto the Hausdorff space
\(\mathscr C\), and hence a homeomorphism. The space \(\mathcal P(\partial_{\mathrm{RP}})\) is a Bauer simplex
whose extreme points are the point masses~\cite{phelps2001choquet}.
Since \(\mathcal B\) is affine, \(\mathscr C\) is also a Bauer simplex,
with extreme points \(\mathcal B(\delta_\xi)=\mu^\xi\) for
\(\xi\in\partial_{\mathrm{RP}}\).
\end{proof}

\section{Root decomposition of boundary points}\label{sec:intrinsic-root-structure}
We recover the positive masses \(p_i\) of the root branches and their
boundary points \(\xi_i\) from the probabilities of sampling trees
whose root has a single child. These probabilities give moments that
determine each distinct mass \(p\) and the finite measure
\(\sum_{i:p_i=p}\mu^{\xi_i}\). By \eqref{eq:canonical-kernel-fibre-support},
its image under \(\Xi\) is \(\sum_{i:p_i=p}\delta_{\xi_i}\), which
counts how many branches of mass \(p\) have each boundary point.

\subsection{Marked mass partitions}
We view the pairs \((p_i,\xi_i)\) as a mass partition with marks in \(\partial_{\mathrm{RP}}\). In the following definition, the marks may lie in any compact metrizable space.

\begin{definition}[Marked mass-partition space]
\label{def:marked-mass-partitions}
For a nonempty compact metrizable space \(K\), let \(\mathfrak M(K)\)
consist of finite Borel measures
\[
    \nu=\sum_{i\in I}p_i^2\delta_{(p_i,z_i)},\qquad
    p_i>0,\quad z_i\in K,\quad \sum_{i\in I}p_i\le1,
\]
where \(I\) is finite or countable. Repeated
pairs are recorded with multiplicity: \(m\) copies of \((p,z)\) give an
atom of weight \(mp^2\). Equip \(\mathfrak M(K)\) with the weak topology
of finite measures on \([0,1]\times K\). The dust mass of \(\nu\) is
\(p_0(\nu)=1-\sum_i p_i\).
\end{definition}

\begin{remark}[Quadratic weights]
\label{rem:dust-fixed-point-topology}
At a given vertex, a box that receives a single label contributes a leaf,
just as a dust label does. Replacing a box by dust can therefore change
the sampled tree only if at least two labels choose it. Two given labels
both choose box \(i\) with probability \(p_i^2\). Thus, for any Borel set
\(A\subseteq[0,1]\times K\), \(\nu(A)\) is the probability that the two
labels choose the same box \(i\) with \((p_i,z_i)\in A\). For \(k\ge2\),
the probability that \(k\) given labels all choose box \(i\) is
\(p_i^2 p_i^{k-2}\), which accounts for the factor \(x^{k-2}\) in
\eqref{eq:marked-moment-coordinates}.

In particular, for \(0<\varepsilon\le1\), the probability that a given
pair chooses the same box of mass at most \(\varepsilon\) is
\[
    \nu((0,\varepsilon]\times K)
    = \sum_{i:p_i\le\varepsilon} p_i^2
    \le \varepsilon\sum_{i:p_i\le\varepsilon} p_i
    \le \varepsilon.
\]
If \(r\ge2\) labels independently choose a box or dust at the vertex,
a union bound over pairs shows that treating all boxes of mass at most
\(\varepsilon\) as dust changes the sampled tree with probability at most
\(\binom{r}{2}\varepsilon\).

For example, take \(\nu_n=n^{-2}\sum_{i=1}^n\delta_{(1/n,z_{n,i})}\),
where the marks \(z_{n,i}\in K\) are arbitrary. Its total mass is \(1/n\),
so \(\nu_n\Rightarrow 0\).  Since \(p_0(\nu_n)=0\) for every \(n\) and
\(p_0(0)=1\), the dust mass is discontinuous at the zero measure.
\end{remark}

\begin{lemma}[Compactness and moment coordinates]
\label{lem:marked-mass-partition-compactness}
The space \(\mathfrak M(K)\) is compact and metrizable, and each
\(\nu\in\mathfrak M(K)\) determines the multiset of pairs \((p_i,z_i)\)
in its representation. Weak convergence is equivalent to convergence of
\begin{equation}\label{eq:marked-moment-coordinates}
    \int x^{k-2}F(z)\,\nu(dx,dz)
       =\sum_i p_i^kF(z_i),
       \qquad k\ge2,\quad F\in C(K,\R).
\end{equation}
\end{lemma}
\begin{proof}
For \(p>0\), the multiplicity of \((p,z)\) is
\(\nu\{(p,z)\}/p^2\). To prove compactness, take a sequence
\(\nu_j=\sum_i(p_i^{(j)})^2\delta_{(p_i^{(j)},z_i^{(j)})}\) and list the masses of each \(\nu_j\) in nonincreasing order. If a list is
finite, extend it by adding zero masses with a fixed mark. Along a
diagonal subsequence we have
\((p_i^{(j)},z_i^{(j)})\to(p_i,z_i)\) for every \(i\). The limiting
masses satisfy \(\sum_i p_i\le1\). For \(g\in C([0,1]\times K)\),
\begin{equation}\label{eq:marked-square-tail-uniform}
    \left|\sum_{i>L}(p_i^{(j)})^2
            g(p_i^{(j)},z_i^{(j)})\right|
       \le \|g\|_\infty p_{L+1}^{(j)}\sum_{i>L}p_i^{(j)}
       \le\frac{\|g\|_\infty}{L+1}.
\end{equation}
Set \(\nu=\sum_{i:p_i>0}p_i^2\delta_{(p_i,z_i)}\in\mathfrak M(K)\).
For fixed \(L\), coordinatewise convergence and continuity of \(g\) imply
that the sum over \(i\le L\) in \(\int g\,d\nu_j\) converges to the
corresponding sum for \(\nu\). By \eqref{eq:marked-square-tail-uniform}, the sums over \(i>L\) in
\(\int g\,d\nu_j\) and \(\int g\,d\nu\) are each bounded in absolute
value by \(\|g\|_\infty/(L+1)\). Hence
\(\limsup_{j\to\infty}\bigl|\int g\,d\nu_j-\int g\,d\nu\bigr|
\le 2\|g\|_\infty/(L+1)\).
Letting \(L\to\infty\) gives \(\nu_j\Rightarrow\nu\) along the chosen
subsequence. Thus \(\mathfrak M(K)\) is sequentially compact. Since the weak topology
on subprobability measures on \([0,1]\times K\) is metrizable,
\(\mathfrak M(K)\) is compact.

The real linear span of the test functions \(x^{k-2}F(z)\) is a unital
algebra that separates points of \([0,1]\times K\). Stone--Weierstrass
and the uniform bound on total mass give the convergence criterion.
\end{proof}

\subsection{Recovering the root data}
Recall that \(\langle t\rangle\) is the tree obtained by adding a new root
above \(t\). For a boundary point \(\xi\), we use the sampling probabilities
\(K_{\langle t\rangle}(\xi)\) to compute weighted moments of its root
branch masses. For \(s\in\T\), set
\(k_s=\max\{2,|s|\}\) and define, for integers \(k\ge k_s\),
\begin{equation}\label{eq:single-branch-observable}
    H_{k,s}(\zeta)
       :=\sum_{t\in\T_k}K_s(t)K_{\langle t\rangle}(\zeta),
       \qquad \zeta\in\Omega_M.
\end{equation}
This is a finite linear combination of sampling coordinates, hence continuous.

\begin{lemma}[Single-branch sampling identity]
\label{lem:single-branch-root-identity}
Let \(\omega\) be a recursive paintbox with sampling class
\(\xi=[\omega]_{\mathrm{samp}}\in\partial_{\mathrm{RP}}\).
For each root branch \(i\) of positive mass \(p_i\), let
\(\xi_i=[\omega_i]_{\mathrm{samp}}\) be the sampling class of the
paintbox on that branch. Then
\begin{equation}\label{eq:single-branch-root-moments}
    H_{k,s}(\xi)=\sum_i p_i^kK_s(\xi_i),\qquad k\ge k_s.
\end{equation}
\end{lemma}
\begin{proof}
In a sample of size \(k+1\), where \(k\ge2\), the root has only one
child precisely when all \(k\) remaining labels enter one positive box. For \(p_i>0\),
the event that all \(k\) labels enter box \(i\) has probability \(p_i^k\).
Conditional on this event, the child subtree has shape law \(M_k^{\xi_i}\).
Summing over \(i\) gives
\[
    K_{\langle t\rangle}(\xi)=\sum_i p_i^kM_k^{\xi_i}(t),
    \qquad t\in\T_k.
\]
Coherence of \((M_m^{\xi_i})_{m\ge1}\) gives
\(\sum_{t\in\T_k}K_s(t)M_k^{\xi_i}(t)=K_s(\xi_i)\)
for \(|s|\le k\).
Multiplying the formula for \(K_{\langle t\rangle}(\xi)\) by \(K_s(t)\),
summing over \(t\in\T_k\), and interchanging the nonnegative sums gives
\eqref{eq:single-branch-root-moments}.
\end{proof}

\begin{example}[Root masses and child marks]
\label{ex:root-masses-child-marks}
Since every two-vertex sample has shape \(C_2\), the cases
\((k,s)=(2,C_2)\) and \((3,C_3)\) of
\eqref{eq:single-branch-root-moments} give
\begin{equation}\label{eq:chain-root-observables}
    K_{C_3}(\xi)=\sum_i p_i^2,
    \qquad
    K_{C_4}(\xi)=\sum_i p_i^3K_{C_3}(\xi_i).
\end{equation}
Thus \(K_{C_3}\) is the probability that the two non-root labels enter
the same positive box, whereas \(K_{C_4}\) also depends on the boundary
points of the corresponding branches. Let \(\xi_{\mathrm{ch}}\)
be the class of an environment with a single child of mass one at every
address, and let \(\xi_{\mathrm{dust}}\) be the pure-dust class. Their samples of size \(m\) are the chain \(C_m\) and the star \(S_m\),
respectively.

Now consider an environment whose root has a single child of mass one,
with pure dust at that child. A sample of size \(m\ge2\) then consists
of the root, its only child, and \(m-2\) leaves attached to that child.
The corresponding sampling class has \(K_{C_3}=1\) and \(K_{C_4}=0\),
whereas both probabilities equal one at \(\xi_{\mathrm{ch}}\).

For two root branches of mass \(1/2\) each, \(K_{C_3}=1/2\) for any
choice of child classes. If both branches have class \(\xi_{\mathrm{ch}}\),
then \(K_{C_4}=1/4\). If one has class \(\xi_{\mathrm{ch}}\) and the
other \(\xi_{\mathrm{dust}}\), then \(K_{C_4}=1/8\).
\end{example}

\begin{proposition}[Identification of marked root masses]
\label{prop:marked-root-identifiability}
A measure
\(\nu=\sum_i p_i^2\delta_{(p_i,\xi_i)}
\in\mathfrak M(\partial_{\mathrm{RP}})\)
is uniquely determined by the weighted moments
\(\sum_i p_i^kK_s(\xi_i)\), for \(s\in\T\) and \(k\ge k_s\).
Consequently, for every \(p>0\) and \(\eta\in\partial_{\mathrm{RP}}\),
the sampling class of a recursive paintbox determines how many of its
root branches have mass \(p\) and sampling class \(\eta\).
\end{proposition}
\begin{proof}
For each \(s\), define the finite measure on \([0,1]\)
\begin{equation}\label{eq:root-one-dimensional-measure}
    \lambda_s=\sum_i p_i^{k_s}K_s(\xi_i)\,\delta_{p_i}.
\end{equation}
Its total mass is at most \(\sum_i p_i^2\le1\). For each \(r\in\Nzero\), the weighted moment with exponent \(k=k_s+r\)
equals the \(r\)th moment of \(\lambda_s\)
\[
    \sum_i p_i^{k_s+r}K_s(\xi_i)
    =\int x^r\,\lambda_s(dx),\qquad r\in\Nzero.
\]
These moments determine \(\lambda_s\), since polynomials are dense in
\(C([0,1],\R)\). Taking \(s=\one\) gives
\(\lambda_\one=\sum_i p_i^2\delta_{p_i}\), since
\(K_\one=1\) and \(k_\one=2\). Its atoms are the distinct positive
root masses, and each such mass \(p\) has multiplicity
\begin{equation}\label{eq:root-equal-mass-multiplicity}
    \#\{i:p_i=p\}
       =\frac{\lambda_\one(\{p\})}{p^2}
       \le\frac1p<\infty.
\end{equation}

To recover the sampling classes of the branches of a fixed positive
root mass \(p\), consider the finite measure
\(Q_p=\sum_{i:p_i=p}\mu^{\xi_i}\) on path space.
For every finite path \(\gamma\) ending at \(s\), we have
\begin{equation}\label{eq:root-equal-mass-barycentre}
    Q_p(C_\gamma)
       =\sum_{i:p_i=p}\frac{K_s(\xi_i)}{u(s)}
       =\frac{\lambda_s(\{p\})}{p^{k_s}u(s)}.
\end{equation}
These cylinder probabilities determine \(Q_p\).
Pushing \(Q_p\) forward under \(\Xi\) and applying
\eqref{eq:canonical-kernel-fibre-support} gives
\[
    \Xi_{\#}Q_p
    =\sum_{i:p_i=p}\Xi_{\#}\mu^{\xi_i}
    =\sum_{i:p_i=p}\delta_{\xi_i}.
\]
For each \(\eta\in\partial_{\mathrm{RP}}\), \((\Xi_{\#}Q_p)(\{\eta\})\)
counts the root branches with mass \(p\) and sampling class \(\eta\).
As \(p\) ranges over the positive root masses, these counts determine
the multiset of pairs \((p_i,\xi_i)\) and hence \(\nu\). For a recursive paintbox with sampling class \(\xi\),
\eqref{eq:single-branch-root-moments} gives
\(\sum_i p_i^k K_s(\xi_i)=H_{k,s}(\xi)\) for every \(s\in\T\) and \(k\ge k_s\).
Thus \(\xi\) determines \(\#\{i:p_i=p,\ \xi_i=\eta\}\)
for every \(p>0\) and \(\eta\in\partial_{\mathrm{RP}}\).
\end{proof}

\subsection{The recursive homeomorphism}
\label{subsec:recursive-fixed-point}
\begin{theorem}[Recursive root decomposition]
\label{thm:recursive-fixed-point-boundary}
Let \(\omega\) represent \(\xi\in\partial_{\mathrm{RP}}\), with root
masses \(p_i\) and child environments \(\omega_i\). The measure
\[
    \mathscr R(\xi)=\sum_{i\ge1:p_i>0}p_i^2
           \delta_{(p_i,[\omega_i]_{\mathrm{samp}})}
\]
is independent of the representative and defines a homeomorphism
\(\mathscr R:\partial_{\mathrm{RP}}\to\mathfrak M(\partial_{\mathrm{RP}})\).
To construct the inverse, start with a new root. For each atom of mass
\(m p^2\) at \((p,z)\), choose a representative of \(z\) and attach
\(m\) copies to the root as branches of mass \(p\).
The remaining mass \(1-\sum_i p_i\) is assigned to dust.
\end{theorem}
\begin{proof}
By \zcref{prop:marked-root-identifiability}, 
the measure \(\mathscr R(\xi)\)
is independent of the choice of representative. Under the sampling rule in \eqref{eq:paintbox-labelled-prob-recursion},
each non-root label independently chooses box \(i\ge1\) with probability
\(p_i\) or falls in the dust with probability \(1-\sum_{i\ge1}p_i\).
Each label in the dust forms a separate leaf. Conditional on these choices,
the subtrees in the occupied boxes are independent. If box \(i\) receives
\(m\ge1\) labels, its subtree has shape law \(M_m^{[\omega_i]_{\mathrm{samp}}}\).
Suppose \(\mathscr R(\xi)=\mathscr R(\xi')\). For every \(p>0\) and \(z\in\partial_{\mathrm{RP}}\), the two paintboxes
have the same number of root branches with mass \(p\) and sampling class \(z\). The sampling rule above therefore gives
\(M_n^\xi=M_n^{\xi'}\) for every \(n\), so \(\xi=\xi'\).

To prove surjectivity, fix a measure in \(\mathfrak M(\partial_{\mathrm{RP}})\).
Start with a new root. For each atom of mass \(mp^2\) at \((p,z)\),
attach \(m\) copies of a representative of \(z\) to the root
and give each of these branches mass \(p\).
By construction, \(\mathscr R\) maps the sampling class of this paintbox
to the prescribed measure. Thus \(\mathscr R\) is surjective.

By \zcref{prop:marked-root-identifiability}, the integrals of
\(x^{k-2}K_s(z)\), for \(s\in\T\) and \(k\ge k_s\), separate the
measures in \(\mathfrak M(\partial_{\mathrm{RP}})\). Since these integrals are weakly continuous and
\(\mathfrak M(\partial_{\mathrm{RP}})\) is compact, convergence of all
these integrals characterizes weak convergence. By
\eqref{eq:single-branch-root-moments},
\begin{equation}\label{eq:root-decomposition-moments}
    \int x^{k-2}K_s(z)\,\mathscr R(\xi)(dx,dz)=H_{k,s}(\xi).
\end{equation}
If \(\xi_j\to\xi\), then \(H_{k,s}(\xi_j)\to H_{k,s}(\xi)\) for every
\(s\in\T\) and \(k\ge k_s\), since \(H_{k,s}\) is continuous by
\eqref{eq:single-branch-observable}. Hence
\(\mathscr R(\xi_j)\Rightarrow\mathscr R(\xi)\), proving continuity. As a
continuous bijection from a compact space to a Hausdorff space,
\(\mathscr R\) is a homeomorphism.
\end{proof}

\begin{corollary}[Convergence of boundary points]
\label{cor:root-decomposition-convergence}
Convergence of boundary points \(\xi_j\to\xi\) is equivalent to
\(\mathscr R(\xi_j)\Rightarrow\mathscr R(\xi)\). It is also characterized
by convergence \(K_{\langle t\rangle}(\xi_j)\to K_{\langle t\rangle}(\xi)\)
for every \(t\in\T\) with \(|t|\ge2\), or by convergence
\(H_{k,s}(\xi_j)\to H_{k,s}(\xi)\) for every \(s\in\T\) and
\(k\ge k_s\).
\end{corollary}
\begin{proof}
By \zcref{thm:recursive-fixed-point-boundary}, \(\xi_j\to\xi\) if and
only if \(\mathscr R(\xi_j)\Rightarrow\mathscr R(\xi)\). The integrals in
\eqref{eq:root-decomposition-moments} characterize weak convergence on
\(\mathfrak M(\partial_{\mathrm{RP}})\), so this is equivalent to
\(H_{k,s}(\xi_j)\to H_{k,s}(\xi)\) for every \(s\in\T\) and \(k\ge k_s\).

By \eqref{eq:single-branch-observable}, each \(H_{k,s}\) is a finite
linear combination of the coordinates \(K_{\langle t\rangle}\) with
\(|t|=k\ge2\). Thus, if \(K_{\langle t\rangle}(\xi_j)\to
K_{\langle t\rangle}(\xi)\) for every \(|t|\ge2\), all the moments
converge to their values at \(\xi\), and hence \(\xi_j\to\xi\).
Conversely, if \(\xi_j\to\xi\), these coordinates converge by continuity.
\end{proof}

\begin{corollary}[Sampling equivalence of recursive paintboxes]
\label{cor:paintbox-sampling-equivalence}
Let \(\omega\) and \(\omega'\) be recursive paintboxes with positive root
branches \((p_i,\omega_i)_{i\in I}\) and \((p'_j,\omega'_j)_{j\in I'}\),
respectively. Then \(\omega\sim_{\mathrm{samp}}\omega'\) if and only if there is a
bijection \(\sigma:I\to I'\) such that
\begin{equation}\label{eq:recursive-paintbox-equivalence}
    p_i=p'_{\sigma(i)}
    \quad\text{and}\quad
    \omega_i\sim_{\mathrm{samp}}\omega'_{\sigma(i)}
    \qquad\text{for every }i\in I.
\end{equation}
In each paintbox, retain only the vertices reachable from the root along
positive-mass branches. Give each retained edge the relative mass of its
branch and forget the child indices. The two paintboxes are sampling
equivalent if and only if the resulting weighted rooted trees are isomorphic.
\end{corollary}
\begin{proof}
By the injectivity of \(\mathscr R\), two paintboxes are sampling equivalent
if and only if their marked root measures agree. These measures agree if and only if there is a bijection
\(\sigma:I\to I'\) satisfying \eqref{eq:recursive-paintbox-equivalence}.

Suppose the two paintboxes are sampling equivalent and choose a bijection
\(\sigma\) as above. For each \(i\), the child paintboxes \(\omega_i\) and
\(\omega'_{\sigma(i)}\) are again sampling equivalent, so the same argument
applies to their positive-mass branches. Continuing level by level gives
an isomorphism of the retained rooted trees that preserves every edge weight.

Conversely, such an isomorphism also preserves the dust mass at each vertex
\(v\), since \(p_{v,0}=1-\sum_{i\ge1}p_{v,i}\). If a label chooses a child
in the first paintbox, let it choose that child's image in the second.
If it chooses dust in the first, let it choose dust in the second as well.
Under this coupling, the two samples have the same shape for every \(n\),
so the paintboxes are sampling equivalent.
\end{proof}

\subsection{Convergence of finite trees}
\label{subsec:finite-tree-root-convergence}
For a finite tree \(t=\langle a_1,\ldots,a_r\rangle\) with \(|t|\ge2\),
we use the relative size of each root branch \(a_i\) as its mass
and its image \(\iota_M(a_i)\in\Omega_M\) as its mark.
The resulting marked root measure is
\begin{equation}\label{eq:finite-tree-marked-root-measure}
    \nu_t=\sum_{i=1}^r
       \left(\frac{|a_i|}{|t|-1}\right)^2
       \delta_{\left(|a_i|/(|t|-1),\,\iota_M(a_i)\right)}
       \in\mathfrak M(\Omega_M).
\end{equation}
We set \(\nu_\one=0\). 
Since \(\partial_{\mathrm{RP}}\) is a closed subspace of
\(\Omega_M\), the inclusion identifies
\(\mathfrak M(\partial_{\mathrm{RP}})\) with a closed subspace of
\(\mathfrak M(\Omega_M)\), both equipped with their weak topologies.

\begin{corollary}[Convergence in terms of marked root measures]
\label{cor:finite-tree-root-convergence}
Let \((t_j)\) be finite trees with \(|t_j|\to\infty\), and let
\(\xi\in\partial_{\mathrm{RP}}\). Then
\begin{equation}\label{eq:finite-tree-root-convergence}
    \iota_M(t_j)\longrightarrow\xi
    \quad\Longleftrightarrow\quad
    \nu_{t_j}\Longrightarrow\mathscr R(\xi)
       \quad\text{in }\mathfrak M(\Omega_M).
\end{equation}
Consequently, under every central measure \(\mu\),
\(\nu_{X_N}\Rightarrow\mathscr R(\Xi)\) almost surely.
\end{corollary}
\begin{proof}
For each finite tree \(a\), write \(\xi_a=q(\omega(a))\) for the sampling
class of its paintbox, constructed just before \zcref{prop:finite-tree-paintbox-approximation}.
For each fixed \(s\in\T\), the estimate
\eqref{eq:finite-tree-paintbox-approximation} shows that
\(K_s(a)-K_s(\xi_a)\to0\) uniformly as \(|a|\to\infty\). Since \(0\le K_s\le 1\), the same convergence holds for every polynomial
in finitely many Martin coordinates. By Stone--Weierstrass, each
\(F\in C(\Omega_M,\R)\) is a uniform limit of such polynomials, so
\[
    \lim_{R\to\infty}\sup_{|a|>R}
    |F(\iota_M(a))-F(\xi_a)|=0.
\]
Since \(|t_j|\to\infty\), the estimate above with \(a=t_j\) shows that
\(\iota_M(t_j)\to\xi\) if and only if \(\xi_{t_j}\to\xi\).

For \(t=\langle a_1,\ldots,a_r\rangle\) of size \(N\ge2\), write
\(p_i=|a_i|/(N-1)\). In \(\omega(t)\), branch \(a_i\) carries the paintbox \(\omega(a_i)\).
Thus \(\mathscr R(\xi_t)\) is obtained from \(\nu_t\) by replacing each
mark \(\iota_M(a_i)\) with \(\xi_{a_i}\) and keeping its weight \(p_i^2\).
The branches with \(|a_i|\le R\) have total weight at most
\(\frac{R}{N-1}\sum_i p_i=R/(N-1)\). Thus for \(k\ge2\) and
\(F\in C(\Omega_M,\R)\),
\[
    \left|\int x^{k-2}F(z)\,
       (\nu_t-\mathscr R(\xi_t))(dx,dz)\right|
       \le \frac{2\|F\|_\infty R}{N-1}
        +\sup_{|a|>R}|F(\iota_M(a))-F(\xi_a)|.
\]
Take \(t=t_j\) in the preceding estimate and let \(j\to\infty\), then
\(R\to\infty\). For every test function in
\eqref{eq:marked-moment-coordinates}, the difference between its integrals
against \(\nu_{t_j}\) and \(\mathscr R(\xi_{t_j})\) tends to zero.
The weak convergence criterion in \zcref{lem:marked-mass-partition-compactness}
therefore gives
\(\nu_{t_j}\Rightarrow\mathscr R(\xi)\) if and only if
\(\mathscr R(\xi_{t_j})\Rightarrow\mathscr R(\xi)\).
By \zcref{cor:root-decomposition-convergence}, the latter is equivalent
to \(\xi_{t_j}\to\xi\), and hence, by the first part of the proof,
to \(\iota_M(t_j)\to\xi\). This proves \eqref{eq:finite-tree-root-convergence}.

Under any central measure, \(\iota_M(X_N)\to\Xi\) almost surely by
\zcref{thm:universal-boundary-variable}. On this event,
\eqref{eq:finite-tree-root-convergence} gives
\(\nu_{X_N}\Rightarrow\mathscr R(\Xi)\).
\end{proof}

\begin{example}[Birth-ordered trees]
For \(h_n\) in \zcref{ex:ordered-boundary-nonextension}, the singleton
branches contribute total weight \((n-1)/(2n-1)^2\to0\) to its marked
root measure. The remaining branch \(C_n\) converges to
\(\xi_{\mathrm{ch}}\), and its relative mass tends to \(1/2\).
For the star \(s_n\), the marked root measure has total weight
\(1/(2n-1)\to0\). Thus
\[
    \nu_{\pi_{\mathrm{ord}}(h_n)}
       \Longrightarrow\tfrac14\delta_{(1/2,\xi_{\mathrm{ch}})},
    \qquad
    \nu_{\pi_{\mathrm{ord}}(s_n)}\Longrightarrow0.
\]
The first limit has a single positive root mass \(1/2\),
marked by \(\xi_{\mathrm{ch}}\), and dust mass \(1/2\).
The second is the pure-dust point \(\xi_{\mathrm{dust}}\).
\end{example}

\section{Extremal decomposition of recursive Ewens tree laws}
\label{sec:ewens-decomposition}
We apply the boundary classification to the central measures associated
with recursive Ewens trees. Their unique extremal decompositions are
obtained from recursive environments with independent Poisson--Dirichlet splits.
We begin with parameters \(\theta,a>0\), using Ewens parameter
\(\theta a\) at the root and \(\theta\) at every other vertex.

\subsection{Ewens partitions and Poisson--Dirichlet paintboxes}
\label{subsec:ewens-pd-background}

For \(x>0\) and \(m\in\Nzero\), write
\(\rising{x}{m}:=\prod_{j=0}^{m-1}(x+j)\), with
\(\rising{x}{0}:=1\).
Fix a set of \(m\) labels. Under the Ewens law with parameter \(\eta>0\),
the probability of a given partition \(\pi=\{B_1,\ldots,B_r\}\) of this set is
\begin{equation}\label{eq:ewens-partition-probability}
    \Prob_\eta(\pi)=\frac{\eta^r}{\rising{\eta}{m}}
    \prod_{j=1}^r (|B_j|-1)!.
\end{equation}
A partition with this law can be generated one label at a time by the
Chinese restaurant process. Given the partition \(\pi\) of the first
\(m\) labels, the next label starts a new block with probability
\(\eta/(\eta+m)\) or joins block \(B_j\) with probability
\(|B_j|/(\eta+m)\) for \(1\le j\le r\).

The Ewens law has a Poisson--Dirichlet paintbox representation.
Let \(V_1,V_2,\ldots\) be independent
\(\operatorname{Beta}(1,\eta)\) variables and set
\[
    \widetilde P_i=V_i\prod_{j<i}(1-V_j),\qquad i\ge1.
\]
The sequence \((\widetilde P_i)\) has the \(\operatorname{GEM}(\eta)\)
law, and its decreasing rearrangement \(P\) has 
Poisson--Dirichlet law. We write \(\PD(\eta)=\PD(0,\eta)\), following
the two-parameter notation of \cite{pitman1997two}.
After \(m\) breaks, the remaining mass is
\(1-\sum_{i=1}^m\widetilde P_i=\prod_{i=1}^m(1-V_i)\).
It decreases with \(m\), and its expectation \((\eta/(\eta+1))^m\)
tends to zero. Thus the dust mass is zero almost surely.
The GEM masses are in size-biased order, whereas \(P\) is ranked
\cite{kingman1975random,donnelly1989continuity}.
Conditionally on \(P\), assign labels independently to boxes with
probabilities \(P_i\). Averaging over \(P\) gives
\eqref{eq:ewens-partition-probability}
\cite{ewens1972sampling,kingman1978representation}. See also
\cite[Chapters~2--3]{pitman2006combinatorial} and \cite{crane2016ewens}.

\subsection{Recursive Ewens trees and hook-product weights}
\label{subsec:ewens-fragmentation-hook-law}
\label{subsec:pochhammer-hook-coherence}

On a finite nonempty ordered label set \(B\), place \(\min B\) at the
root. If \(|B|=1\), stop. Otherwise, partition
\(B\setminus\{\min B\}\) according to the Ewens law, with parameter
\(\theta a\) at the global root and \(\theta\) at every other vertex.
Each block is the label set of a child subtree, and the construction
continues independently inside the blocks. We call the resulting tree
a \emph{recursive Ewens tree} with parameters \((\theta,a)\).
Write \(\widehat T_n\) for the resulting quotient-labelled tree on
\(\{0,1,\ldots,n-1\}\).

Height asymptotics for \(a=1\) were studied in
\cite{zhang2026height}, where the recursive construction and its
identification with the Plancherel law at \(\theta=2\) are given in
Section~2.1 and Theorem~2.4, respectively.
When \(a=1\), recording the birth order of siblings gives the Chinese
restaurant tree with parameters \((0,\theta)\) of
\cite[Section~6.3 and Remark~6.4]{evans2012trickle}.

The finite laws can be written in terms of the hook lengths. For
\(t\in\T\) and \(n\ge1\), set
\[
\begin{aligned}
    d_{\theta,a}(t)
      &:=|t|!\,a^{\deg_t^+(\rootv(t))}
        \prod_{v\in V(t)}
        \frac{(h_v(t)-1)!}{\rising{\theta}{h_v(t)-1}},\\
    Z_n(\theta,a)
      &:=\frac{n!(n-1)!}{\theta^{n-1}}
        \frac{\rising{(\theta a)}{n-1}}{\rising{\theta}{n-1}}.
\end{aligned}
\]
Here \(\rootv(t)\) is the root of \(t\), and \(\deg_t^+(v)\) is the
number of children of \(v\) in \(t\).

\begin{proposition}[The hook-product law]
\label{prop:ewens-fragmentation-equals-hook}
For \(\theta,a>0\), \(n\ge1\), \(t\in\T_n\), and \(L\in\QSL(t)\),
the recursive Ewens tree on \(\{0,1,\ldots,n-1\}\) satisfies
\[
    \Prob\{\widehat T_n=(t,L)\}
       =\frac{d_{\theta,a}(t)}{Z_n(\theta,a)}.
\]
Its shape law is therefore
\[
    M_n^{(\theta,a)}(t)
       :=\frac{u(t)d_{\theta,a}(t)}{Z_n(\theta,a)},\qquad t\in\T_n.
\]
In particular,
\(Z_n(\theta,a)=\sum_{t\in\T_n}u(t)d_{\theta,a}(t)\).
\end{proposition}

\begin{proof}
Fix a quotient labelling \(L\) of \(t\). At a non-root vertex \(v\), the
partition probability is
\[
    \frac{\theta^{\deg_t^+(v)}}{\rising{\theta}{h_v(t)-1}}
       \prod_{c\text{ child of }v}(h_c(t)-1)!.
\]
At the root, replace \(\theta\) by \(\theta a\). 
The total exponent of \(\theta\) in the numerator is
\(\sum_{v\in V(t)}\deg_t^+(v)=n-1\), and the root contributes the
additional factor \(a^{\deg_t^+(\rootv(t))}\). Thus
\[
    \Prob\{\widehat T_n=(t,L)\}
       =\frac{a^{\deg_t^+(\rootv(t))}\theta^{n-1}}
               {\rising{(\theta a)}{n-1}}
         \prod_{v\ne\rootv(t)}
            \frac{(h_v(t)-1)!}{\rising{\theta}{h_v(t)-1}}
       =\frac{d_{\theta,a}(t)}{Z_n(\theta,a)}.
\]
For \(n=1\), the products are empty and
\(d_{\theta,a}(\one)=Z_1(\theta,a)=1\).
All \(u(t)\) quotient labellings of \(t\) have the same probability, so
\(M_n^{(\theta,a)}(t)=u(t)d_{\theta,a}(t)/Z_n(\theta,a)\).
These shape probabilities sum to one over \(t\in\T_n\), giving
\(Z_n(\theta,a)=\sum_{t\in\T_n}u(t)d_{\theta,a}(t)\).
\end{proof}

The ratio of successive normalizing constants is
\[
    \Lambda_n(\theta,a)
      :=\frac{Z_{n+1}(\theta,a)}{Z_n(\theta,a)}
       =\frac{n(n+1)}{\theta}\frac{n+\theta a-1}{n+\theta-1},
       \qquad n\ge1.
\]

\begin{proposition}[Coherence of the hook-product weights]
\label{prop:pochhammer-hook-centrality}
For every \(n\ge1\) and \(t\in\T_n\),
\[
    \sum_{T:t\nearrow T}n(t;T)d_{\theta,a}(T)
       =\Lambda_n(\theta,a)d_{\theta,a}(t).
\]
Consequently, the laws \((M_n^{(\theta,a)})_{n\ge1}\) form a coherent
family.
\end{proposition}

\begin{proof}
For \(k\ge1\), set \(q_\theta(k)=k/(k+\theta-1)\).
Attaching a leaf at \(v\in V(t)\) increases the hook length of \(x\)
precisely when \(x\preceq_t v\), so
\[
    \frac{d_{\theta,a}(t^{+v})}{d_{\theta,a}(t)}
       =(n+1)a^{\ind_{\{v=\rootv(t)\}}}
          \prod_{x\preceq_t v}q_\theta(h_x(t)).
\]
Since \(n(t;T)\) counts the vertices \(v\in V(t)\) for which
\(t^{+v}\cong T\), it suffices to show that
\(\sum_{v\in V(t)} d_{\theta,a}(t^{+v})/d_{\theta,a}(t)
=\Lambda_n(\theta,a)\).

We prove the identity for all \(a>0\) by induction on \(n\).
For \(n=1\), the sum consists of the single term
\(2a/\theta=\Lambda_1(\theta,a)\). Suppose \(n\ge2\).
Attaching a leaf at the root contributes \((n+1)a q_\theta(n)\).

Let \(u\) be a root subtree of size \(m\). For \(v\in V(u)\), the root
of \(t\) is the only ancestor of \(v\) outside \(u\). Thus
\[
    \frac{d_{\theta,a}(t^{+v})}{d_{\theta,a}(t)}
    =(n+1)q_\theta(n)\prod_{x\preceq_u v}q_\theta(h_x(u))
    =\frac{(n+1)q_\theta(n)}{m+1}
       \frac{d_{\theta,1}(u^{+v})}{d_{\theta,1}(u)}.
\]
Summing over \(v\in V(u)\) and applying the induction hypothesis to \(u\)
with \(a=1\), we find that \(u\) contributes
\[
    \frac{(n+1)q_\theta(n)}{m+1}\,\Lambda_m(\theta,1)
    =(n+1)q_\theta(n)\frac{m}{\theta}.
\]
Since the root subtrees have total size \(n-1\), adding their
contributions to the root term gives
\[
    \sum_{v\in V(t)}
    \frac{d_{\theta,a}(t^{+v})}{d_{\theta,a}(t)}
    =(n+1)q_\theta(n)\left(a+\frac{n-1}{\theta}\right)
    =\Lambda_n(\theta,a).
\]
Multiplying the identity for the weights by
\(u(t)/Z_{n+1}(\theta,a)\) and using
\(\Lambda_n(\theta,a)=Z_{n+1}(\theta,a)/Z_n(\theta,a)\)
gives the coherence relation \eqref{eq:coherent-family-definition}.
\end{proof}

The coherent family \((M_n^{(\theta,a)})_{n\ge1}\) determines a unique
central measure \(\mu^{(\theta,a)}\). For \(t\in\T_n\), its harmonic
function on the Hoffman graph is
\(\varphi(t)=M_n^{(\theta,a)}(t)/u(t)=d_{\theta,a}(t)/Z_n(\theta,a)\)
by \zcref{prop:ewens-fragmentation-equals-hook}. Substituting this expression for \(\varphi\) into
\eqref{eq:central-forward-transition} gives
\begin{equation}\label{eq:ewens-forward-transition}
    \Prob_{\mu^{(\theta,a)}}(X_{n+1}=T\mid X_n=t)
       =n(t;T)\frac{\varphi(T)}{\varphi(t)}
       =\frac{n(t;T)d_{\theta,a}(T)}
              {\Lambda_n(\theta,a)d_{\theta,a}(t)},
       \qquad t\nearrow T.
\end{equation}

As in \zcref{ex:uniform-recursive-chain}, the Doob--Martin kernel from
\(s\in\T_n\) to \(t\in\T_N\), with \(n\le N\), is
\(K_s(t)/M_n^{(\theta,a)}(s)\) by \eqref{eq:endpoint-martin-general}.
For fixed \(s\), the denominator is positive and independent of \(t\),
so the Martin compactification is \(\Omega_M\) for every \(\theta,a>0\).

\subsection{The recursive Poisson--Dirichlet mixture}
\label{subsec:recursive-pd-mixture}
Let \(\Omega_{\theta,a}=(P_w)_{w\in\Ulam}\) be a random element of the
compact environment space \(\Delta_\downarrow^{\Ulam}\), with
independent coordinates \(P_w=(P_{w,i})_{i\ge1}\) distributed as
\[
    P_\varnothing\sim\PD(\theta a),\qquad
    P_w\sim\PD(\theta)\quad(w\ne\varnothing).
\]
Let \(\nu_{\theta,a}\) be the law of \(q(\Omega_{\theta,a})\), where
\(q(\omega)=[\omega]_{\mathrm{samp}}\in\partial_{\mathrm{RP}}\) denotes
the sampling class of \(\omega\).

\begin{theorem}[Poisson--Dirichlet extremal decomposition]
\label{thm:ewens-decomposition}
Let \(\mathbf X\) be the infinite path determined by the nested
quotient-labelled samples from \(\Omega_{\theta,a}\), as in
\zcref{prop:paintbox-central}. After averaging over the environment,
the labelled sample on any finite nonempty ordered set \(B\) has the
recursive Ewens tree law with parameters \((\theta,a)\). On path space,
\begin{equation}\label{eq:ewens-boundary-mixture}
    \mathcal L(\mathbf X\mid\Omega_{\theta,a})=\mu^{q(\Omega_{\theta,a})},
    \qquad
    \mathcal L(\mathbf X)=\mu^{(\theta,a)}=\int\mu^\xi\,\nu_{\theta,a}(d\xi).
\end{equation}
Here \(\nu_{\theta,a}\) is the unique probability measure on
\(\partial_{\mathrm{RP}}\) giving this decomposition.

The boundary limit of \(\mathbf X\) satisfies
\begin{equation}\label{eq:ewens-environment-path-boundary-agreement}
    \Xi(\mathbf X)=q(\Omega_{\theta,a})\quad\text{almost surely}.
\end{equation}
Thus \(q(\Omega_{\theta,a})\) generates the tail \(\sigma\)-field of
\(\mathbf X\) modulo null sets.
\end{theorem}
\begin{proof}
Take the uniform array \((U_{m,w})\) to be independent of
\(\Omega_{\theta,a}\). For any finite path \(\gamma\), the event \(\{\mathbf X\in C_\gamma\}\)
is a countable union of events, each determined by finitely many box choices. Each such event is Borel by
\eqref{eq:uniform-routing-realization}, so \(\mathbf X\) is measurable.

Conditional on \(\Omega_{\theta,a}\), the uniform array has its original
product law. Hence, for \(\gamma\) ending at \(s\in\T_n\),
\zcref{prop:paintbox-central} gives
\[
    \Prob\{\mathbf X\in C_\gamma\mid\Omega_{\theta,a}\}
    =\frac{M_n^{\Omega_{\theta,a}}(s)}{u(s)}
    =\mu^{q(\Omega_{\theta,a})}(C_\gamma).
\]
Since the path cylinders determine the law, this proves the first
identity in \eqref{eq:ewens-boundary-mixture}.

To identify the mixture in \eqref{eq:ewens-boundary-mixture} with
\(\mu^{(\theta,a)}\), it suffices to determine the laws of the finite
labelled samples. Let \(B\) be a finite nonempty
ordered label set. We prove by induction on \(|B|\) that, for every
\(a>0\), the unconditional law of the sample on \(B\) is the recursive
Ewens tree law with parameters \((\theta,a)\). For \(|B|=1\),
both constructions give the one-vertex tree.

Suppose \(|B|\ge2\). The root split \(P_\varnothing\sim\PD(\theta a)\)
induces an Ewens partition \(\pi\) of \(B\setminus\{\min B\}\) with
parameter \(\theta a\). On \(\{|\pi|=r\}\), list the blocks
\(B_1,\ldots,B_r\) in order of their least labels, and let \(I_j\) be
the index of the box occupied by \(B_j\). Given the root choices, these
indices are fixed and distinct, and the environments in the selected
boxes are independent copies of \(\Omega_{\theta,1}\). Since their
joint law depends only on \(r\), we have
\begin{equation}\label{eq:pd-descendant-independence}
    \mathcal L\left(
        \bigl((P_{I_jw})_{w\in\Ulam}\bigr)_{j=1}^r
        \,\middle|\,\pi\right)
    =\mathcal L(\Omega_{\theta,1})^{\otimes r}.
\end{equation}
The boxes also use independent uniform arrays, so the subtrees on
\(B_1,\ldots,B_r\) are independent given \(\pi\). Since \(|B_j|<|B|\),
the induction hypothesis with \(a=1\) gives the recursive Ewens tree
law with parameters \((\theta,1)\) on each \(B_j\). Attaching these
subtrees to the root labelled \(\min B\) gives the recursive Ewens tree
on \(B\) with parameters \((\theta,a)\).

For any finite path \(\gamma\) ending at \(t\in\T_n\), averaging the
conditional path law gives
\[
    \int_{\partial_{\mathrm{RP}}}\mu^\xi(C_\gamma)\,\nu_{\theta,a}(d\xi)
       =\Prob\{\mathbf X\in C_\gamma\}
       =\frac{d_{\theta,a}(t)}{Z_n(\theta,a)}
       =\mu^{(\theta,a)}(C_\gamma).
\]
The middle equality follows from \zcref{prop:ewens-fragmentation-equals-hook}
applied to the quotient labelling of \(t\) corresponding to \(\gamma\).
Since the path cylinders determine the law, this proves the mixture
identity in \eqref{eq:ewens-boundary-mixture}.

By \eqref{eq:canonical-kernel-fibre-support}, the boundary limit under
\(\mu^\xi\) is \(\xi\) almost surely. Applying this to the conditional
path law in \eqref{eq:ewens-boundary-mixture} gives
\(\Xi(\mathbf X)=q(\Omega_{\theta,a})\) almost surely. Thus
\(\Xi(\mathbf X)\) has law \(\nu_{\theta,a}\). The tail-field assertion
and the uniqueness of this mixing measure now follow from
parts~(iii) and~(iv) of \zcref{thm:universal-boundary-variable}.
\end{proof}

\begin{corollary}[Recovery of the ranked Poisson--Dirichlet environment]
\label{cor:pd-environment-recovery}
Let \((X_N)_{N\ge1}\) be the shapes of the nested samples from
\(\Omega_{\theta,a}\) in \zcref{thm:ewens-decomposition}. Then
\begin{equation}\label{eq:pd-empirical-environment-recovery}
    \omega(X_N)\longrightarrow\Omega_{\theta,a}
    \quad\text{coordinatewise almost surely}.
\end{equation}
Moreover, \(\Omega_{\theta,a}\) is almost surely a Borel function of
\(\Xi(\mathbf X)\). The fields \(\sigma(\Omega_{\theta,a})\), \(\sigma(\Xi(\mathbf X))\),
and the tail fields of the shape process \((X_N)_{N\ge1}\) and the
full path \(\mathbf X\) all have the same completion under the joint
law of \((\Omega_{\theta,a},\mathbf X)\).
\end{corollary}

\begin{proof}
Let \(G\) be the Borel set of environments in \(\Delta_\downarrow^{\Ulam}\)
such that, for every \(w\in\Ulam\), \(\sum_{i\ge1}p_{w,i}=1\) and
\(p_{w,i}>p_{w,i+1}>0\) for all \(i\ge1\). By the stick-breaking
construction, almost surely the masses at every vertex of
\(\Omega_{\theta,a}\) are positive and sum to one. 
Fix a vertex and write \((\widetilde P_i)_{i\ge1}\)
for its masses in stick-breaking order. For \(i<j\), condition on
\(V_1,\ldots,V_{j-1}\). Then \(\widetilde P_i\) is fixed, and
\(\widetilde P_j=V_j\prod_{k<j}(1-V_k)\) has a density because the
product is positive and \(V_j\) retains its beta law. Hence
\(\Prob\{\widetilde P_i=\widetilde P_j\mid V_1,\ldots,V_{j-1}\}=0\).
Since there are only countably many vertices and pairs \(i<j\), almost
surely the masses at every vertex are pairwise distinct. Thus
\(\Prob\{\Omega_{\theta,a}\in G\}=1\).

Fix \(\omega\in G\) and let \(\eta\in\Delta_\downarrow^{\Ulam}\) satisfy
\(q(\eta)=q(\omega)\). By \zcref{cor:paintbox-sampling-equivalence},
there is a rooted isomorphism \(f\) from the positive-mass subtree of
\(\omega\) to that of \(\eta\) that preserves every edge weight.
Since all branch masses in \(\omega\) are positive, \(f\) is defined
at every address \(w\in\Ulam\). Both environments rank children by mass,
and the child masses at each vertex of \(\omega\) are distinct. Hence
\(f(wi)=f(w)i\) for every \(w\in\Ulam\) and \(i\ge1\).
Since \(f\) fixes the root, induction on depth gives \(f(w)=w\) for
every \(w\). The two environments therefore assign the same mass to
every edge, so \(\eta=\omega\). This proves
\(q^{-1}(\{q(\omega)\})=\{\omega\}\).

By \zcref{thm:universal-boundary-variable,thm:ewens-decomposition},
\(X_N\to\Xi(\mathbf X)=q(\Omega_{\theta,a})\) almost surely.
For each fixed \(s\), \eqref{eq:finite-tree-paintbox-approximation}
gives \(K_s(q(\omega(X_N)))-K_s(X_N)\to0\), so
\(q(\omega(X_N))\to q(\Omega_{\theta,a})\) almost surely.
By continuity of \(q\), every subsequential limit \(\eta\) of
\(\omega(X_N)\) satisfies \(q(\eta)=q(\Omega_{\theta,a})\).
On \(\{\Omega_{\theta,a}\in G\}\), the preceding paragraph then gives
\(\eta=\Omega_{\theta,a}\). Since \(\Delta_\downarrow^{\Ulam}\) is
compact, the whole sequence converges almost surely to
\(\Omega_{\theta,a}\), proving \eqref{eq:pd-empirical-environment-recovery}.

Since \(q|_G\) is a Borel injection, the Lusin--Souslin theorem
\cite[Corollary~15.2]{kechris1995classical} implies that \(q(G)\) is
Borel and that the inverse of \(q|_G\) is Borel on \(q(G)\). Extend this inverse to a Borel map \(r\) by assigning a fixed
environment to every point outside \(q(G)\).
Almost surely, \(\Omega_{\theta,a}\in G\) and, by
\eqref{eq:ewens-environment-path-boundary-agreement},
\(\Xi(\mathbf X)=q(\Omega_{\theta,a})\), so
\(r(\Xi(\mathbf X))=\Omega_{\theta,a}\). Thus \(\Omega_{\theta,a}\)
and \(\Xi(\mathbf X)\) generate the same completed \(\sigma\)-field
under the joint law. Since the path marginal is \(\mu^{(\theta,a)}\),
part~\textup{(iii)} of \zcref{thm:universal-boundary-variable} shows
that the tail fields of the shape process and of the full path also
have this completion.
\end{proof}

\subsection{Atomless mixing measures and special cases}
\label{subsec:ewens-specializations}

\begin{proposition}[Atomless mixing measures]
\label{prop:ewens-not-extreme}
For every \(\theta,a>0\), the mixing measure \(\nu_{\theta,a}\)
has no atoms. In particular, the central measure
\(\mu^{(\theta,a)}\) is not extremal.
\end{proposition}

\begin{proof}
Fix \(\xi\in\partial_{\mathrm{RP}}\). By
\zcref{prop:marked-root-identifiability}, \(\xi\) determines the
at most countable set \(A_\xi\) of its positive root masses.
Realize the root split of \(\Omega_{\theta,a}\) by ranking a
\(\operatorname{GEM}(\theta a)\) sequence. The first mass \(V_1\)
of this sequence has law \(\operatorname{Beta}(1,\theta a)\).
Since \(V_1\) is one of the ranked root masses,
\(q(\Omega_{\theta,a})=\xi\) implies \(V_1\in A_\xi\).
The beta law has a density, so
\[
    \nu_{\theta,a}(\{\xi\})
      =\Prob\{q(\Omega_{\theta,a})=\xi\}
      \le \Prob\{V_1\in A_\xi\}=0.
\]
An extremal central measure has a point mass as its unique mixing
measure, so \zcref{thm:ewens-decomposition} implies that
\(\mu^{(\theta,a)}\) is not extremal.
\end{proof}

\begin{example}[Root-weighted recursive trees]
\label{ex:uniform-recursive-pd-specialization}
Set \(\theta=1\) and fix \(a>0\). For \(t\in\T_n\) with root degree
\(r\) and \(L\in\QSL(t)\), \zcref{prop:ewens-fragmentation-equals-hook}
gives
\[
    \Prob\{\widehat T_n=(t,L)\}
       =\frac{M_n^{(1,a)}(t)}{u(t)}
       =\frac{a^r}{\rising{a}{n-1}}.
\]
Adding label \(n\) increases the root degree \(r\) by one precisely
when its parent is the root. Since
\(\rising{a}{n}=(a+n-1)\rising{a}{n-1}\), label \(n\) attaches to
the root with probability \(a/(a+n-1)\) and to each other vertex
with probability \(1/(a+n-1)\). This is the \((0,a)\)-recursive tree of Dong,
Goldschmidt and Martin \cite[Sections~5--6]{dong2006coagulation}.

By \zcref{thm:ewens-decomposition}, the mixing measure is
\(\mathcal L(q(\Omega_{1,a}))\). When \(a=1\), each of the \(n\) vertices is chosen as the parent with
probability \(1/n\), giving the uniform recursive tree of
\zcref{ex:uniform-recursive-chain}.
\end{example}

\begin{proposition}[Agreement with Fulman's rooted-tree Plancherel law]
\label{prop:fulman-plancherel-agreement}
For \(n\ge1\) and \(t\in\T_n\),
\[
    M_n^{(2,1)}(t)
       =\frac{u(t)d(t)}{\prod_{k=2}^n\binom{k}{2}}
       =\frac{n\,2^{n-1}}
              {|\Aut(t)|\prod_{v\in V(t)}h_v(t)^2}.
\]
This is Fulman's Plancherel measure on unlabelled rooted trees
\cite[Section~3, Definition~1]{fulman2009mixing}.
The forward transition probabilities of \(\mu^{(2,1)}\) agree with
Fulman's upward transitions
\cite[Section~3, Definition~2]{fulman2009mixing}:
\begin{equation}\label{eq:fulman-up-transition}
    \Prob_{\mu^{(2,1)}}(X_{n+1}=T\mid X_n=t)
       =\frac{n(t;T)d(T)}{\binom{n+1}{2}d(t)},
    \qquad t\nearrow T.
\end{equation}
\end{proposition}

\begin{proof}
Since \((h-1)!/\rising{2}{h-1}=1/h\), the hook formula \eqref{eq:rooted-hook} gives
\(d_{2,1}(t)=d(t)\), and
\(Z_n(2,1)=n!(n-1)!/2^{n-1}=\prod_{k=2}^n\binom{k}{2}\).
Substituting \(u(t)=d(t)/|\Aut(t)|\) and
the hook formula gives the second equality
in the formula for \(M_n^{(2,1)}(t)\).
Finally, \(\Lambda_n(2,1)=\binom{n+1}{2}\), so
\eqref{eq:ewens-forward-transition} gives
\eqref{eq:fulman-up-transition}.
\end{proof}
The central measure \(\mu^{\mathrm{Pl}}:=\mu^{(2,1)}\) has the extremal
decomposition
\[
    \mu^{\mathrm{Pl}}
       =\int_{\partial_{\mathrm{RP}}}\mu^\xi\,\nu_{2,1}(d\xi).
\]
This mixing measure has no atoms by \zcref{prop:ewens-not-extreme}.

\printbibliography[title={References}]

@article{bratteli1972inductive,
  title={Inductive limits of finite dimensional {$C^*$}-algebras},
  author={Bratteli, Ola},
  journaltitle={Transactions of the American Mathematical Society},
  volume={171},
  pages={195--234},
  year={1972},
  doi={10.1090/S0002-9947-1972-0312282-2}
}

@book{effros1981dimensions,
  author={Effros, Edward G.},
  title={Dimensions and {$C^*$}-Algebras},
  series={CBMS Regional Conference Series in Mathematics},
  number={46},
  publisher={American Mathematical Society},
  location={Providence, RI},
  year={1981},
  doi={10.1090/cbms/046}
}

@article{hoffman2003combinatorics,
  title={Combinatorics of Rooted Trees and {Hopf} Algebras},
  author={Hoffman, Michael E.},
  journaltitle={Transactions of the American Mathematical Society},
  volume={355},
  number={9},
  pages={3795--3811},
  year={2003},
  doi={10.1090/S0002-9947-03-03317-8}
}

@book{woess2009denumerable,
  title={Denumerable Markov Chains: Generating Functions, Boundary Theory, Random Walks on Trees},
  author={Woess, Wolfgang},
  series={EMS Textbooks in Mathematics},
  year={2009},
  publisher={European Mathematical Society},
  location={Z{\"u}rich},
  pagetotal={xvii+351},
  doi={10.4171/071}
}

@book{kallenberg2005probabilistic,
  title={Probabilistic Symmetries and Invariance Principles},
  author={Kallenberg, Olav},
  series={Probability and Its Applications},
  publisher={Springer},
  location={New York},
  year={2005},
  pagetotal={xii+512},
  doi={10.1007/0-387-28861-9}
}

@book{kechris1995classical,
  title={Classical Descriptive Set Theory},
  author={Kechris, Alexander S.},
  series={Graduate Texts in Mathematics},
  number={156},
  publisher={Springer},
  location={New York},
  year={1995},
  doi={10.1007/978-1-4612-4190-4}
}

@article{kingman1978representation,
  title={The Representation of Partition Structures},
  author={Kingman, John F. C.},
  journaltitle={Journal of the London Mathematical Society. Second Series},
  volume={18},
  number={2},
  pages={374--380},
  year={1978},
  doi={10.1112/jlms/s2-18.2.374}
}

@book{pitman2006combinatorial,
  title={Combinatorial Stochastic Processes},
  author={Pitman, Jim},
  series={Lecture Notes in Mathematics},
  number={1875},
  publisher={Springer},
  location={Berlin},
  year={2006},
  pagetotal={x+260},
  note={Lectures from the 32nd Summer School on Probability Theory held in Saint-Flour, 2002},
  doi={10.1007/b11601500}
}

@article{forman2018representation,
  title={A representation of exchangeable hierarchies by sampling from random real trees},
  author={Forman, Noah and Haulk, Chris and Pitman, Jim},
  journaltitle={Probability Theory and Related Fields},
  volume={172},
  number={1--2},
  pages={1--29},
  year={2018},
  doi={10.1007/s00440-017-0799-4}
}

@article{fulman2009mixing,
  title={Mixing time for a random walk on rooted trees},
  author={Fulman, Jason},
  journaltitle={Electronic Journal of Combinatorics},
  volume={16},
  number={1},
  eid={R139},
  note={13 pp.},
  year={2009},
  doi={10.37236/228}
}

@inproceedings{bergeron1992varieties,
  title={Varieties of increasing trees},
  author={Bergeron, Fran{\c{c}}ois and Flajolet, Philippe and Salvy, Bruno},
  editor={Raoult, J.-C.},
  booktitle={Trees in Algebra and Programming---CAAP '92},
  series={Lecture Notes in Computer Science},
  number={581},
  pages={24--48},
  publisher={Springer},
  year={1992},
  doi={10.1007/3-540-55251-0_2}
}

@article{sagan2008probabilistic,
  title={Probabilistic proofs of hook length formulas involving trees},
  author={Sagan, Bruce E.},
  journaltitle={S{\'e}minaire Lotharingien de Combinatoire},
  volume={61A},
  eid={B61Ab},
  note={10 pp.},
  year={2009},
  url={https://www.mat.univie.ac.at/~slc/wpapers/s61Asagan.html}
}

@article{dynkin1978sufficient,
  title={Sufficient statistics and extreme points},
  author={Dynkin, E. B.},
  journaltitle={Annals of Probability},
  volume={6},
  number={5},
  pages={705--730},
  year={1978},
  doi={10.1214/aop/1176995424}
}

@article{vershik2014central,
  author={Vershik, A. M.},
  title={The problem of describing central measures on the path spaces of graded graphs},
  journaltitle={Functional Analysis and Its Applications},
  volume={48},
  number={4},
  pages={256--271},
  year={2014},
  doi={10.1007/s10688-014-0069-5}
}

@article{kingman1975random,
  title={Random Discrete Distributions},
  author={Kingman, John F. C.},
  journaltitle={Journal of the Royal Statistical Society: Series B (Methodological)},
  volume={37},
  number={1},
  pages={1--15},
  year={1975},
  doi={10.1111/j.2517-6161.1975.tb01024.x}
}

@article{gnedin2006recursive,
  title={Recursive Partition Structures},
  author={Gnedin, Alexander V. and Yakubovich, Yuri},
  journaltitle={Annals of Probability},
  volume={34},
  number={6},
  pages={2203--2218},
  year={2006},
  doi={10.1214/009117906000000584}
}

@article{ewens1972sampling,
  title={The sampling theory of selectively neutral alleles},
  author={Ewens, Warren J.},
  journaltitle={Theoretical Population Biology},
  volume={3},
  number={1},
  pages={87--112},
  year={1972},
  doi={10.1016/0040-5809(72)90035-4}
}

@article{pitman1997two,
  title={The two-parameter {Poisson--Dirichlet} distribution derived from a stable subordinator},
  author={Pitman, Jim and Yor, Marc},
  journaltitle={Annals of Probability},
  volume={25},
  number={2},
  pages={855--900},
  year={1997},
  doi={10.1214/aop/1024404422}
}

@article{evans2012trickle,
  title={Trickle-down processes and their boundaries},
  author={Evans, Steven N. and Gr{\"u}bel, Rudolf and Wakolbinger, Anton},
  journaltitle={Electronic Journal of Probability},
  volume={17},
  pages={1--58},
  year={2012},
  eid={1},
  doi={10.1214/EJP.v17-1698},
  eprint={1010.0453},
  eprinttype={arxiv},
  eprintclass={math.PR}
}

@article{evans2017remy,
  title={{Doob--Martin} boundary of {R}{\'e}my's tree growth chain},
  author={Evans, Steven N. and Gr{\"u}bel, Rudolf and Wakolbinger, Anton},
  journaltitle={Annals of Probability},
  volume={45},
  number={1},
  pages={225--277},
  year={2017},
  doi={10.1214/16-AOP1112}
}

@article{grubel2015recursive,
  title={Random recursive trees: A boundary theory approach},
  author={Gr{\"u}bel, Rudolf and Michailow, Igor},
  journaltitle={Electronic Journal of Probability},
  volume={20},
  pages={1--22},
  year={2015},
  eid={37},
  doi={10.1214/EJP.v20-3832},
  eprint={1406.7614},
  eprinttype={arxiv},
  eprintclass={math.PR}
}

@article{forman2020mass,
  title={Exchangeable hierarchies and mass-structure of weighted real trees},
  author={Forman, Noah},
  journaltitle={Electronic Journal of Probability},
  volume={25},
  pages={1--28},
  year={2020},
  eid={131},
  doi={10.1214/20-EJP522}
}

@article{gerstenberg2020interval,
  title={Exchangeable interval hypergraphs and limits of ordered discrete structures},
  author={Gerstenberg, Julian},
  journaltitle={Annals of Probability},
  volume={48},
  number={3},
  pages={1128--1167},
  year={2020},
  doi={10.1214/19-AOP1384}
}

@article{pitman2014regenerative,
  author={Pitman, Jim and Rizzolo, Douglas and Winkel, Matthias},
  title={Regenerative tree growth: structural results and convergence},
  journaltitle={Electronic Journal of Probability},
  volume={19},
  pages={1--27},
  year={2014},
  eid={70},
  doi={10.1214/EJP.v19-3040}
}

@article{crane2016ewens,
  author={Crane, Harry},
  title={The ubiquitous {Ewens} sampling formula},
  journaltitle={Statistical Science},
  volume={31},
  number={1},
  pages={1--19},
  year={2016},
  doi={10.1214/15-STS529}
}

@article{donnelly1989continuity,
  author={Donnelly, Peter and Joyce, Paul},
  title={Continuity and weak convergence of ranked and size-biased permutations on the infinite simplex},
  journaltitle={Stochastic Processes and their Applications},
  volume={31},
  number={1},
  pages={89--103},
  year={1989},
  doi={10.1016/0304-4149(89)90104-X}
}

@article{gnedin2006zigzag,
  author={Gnedin, Alexander and Olshanski, Grigori},
  title={Coherent permutations with descent statistic and the boundary problem for the graph of zigzag diagrams},
  journaltitle={International Mathematics Research Notices},
  volume={2006},
  eid={51968},
  note={39 pp.},
  year={2006},
  doi={10.1155/IMRN/2006/51968}
}

@article{tarrago2018zigzag,
  author={Tarrago, Pierre},
  title={Zigzag diagrams and {Martin} boundary},
  journaltitle={Annals of Probability},
  volume={46},
  number={5},
  pages={2562--2620},
  year={2018},
  doi={10.1214/17-AOP1234}
}

@book{phelps2001choquet,
  author={Phelps, Robert R.},
  title={Lectures on {Choquet's} Theorem},
  edition={2},
  series={Lecture Notes in Mathematics},
  number={1757},
  publisher={Springer},
  location={Berlin},
  pagetotal={x+130},
  year={2001},
  doi={10.1007/b76887}
}

@article{gnedin2020nested,
  title={On nested infinite occupancy scheme in random environment},
  author={Gnedin, Alexander and Iksanov, Alexander},
  journaltitle={Probability Theory and Related Fields},
  volume={177},
  number={3--4},
  pages={855--890},
  year={2020},
  doi={10.1007/s00440-020-00963-0}
}

@article{devroye1998universal,
  author={Devroye, Luc},
  title={Universal Limit Laws for Depths in Random Trees},
  journaltitle={SIAM Journal on Computing},
  volume={28},
  number={2},
  pages={409--432},
  year={1998},
  doi={10.1137/S0097539795283954}
}

@article{janson2019split,
  author={Janson, Svante},
  title={Random Recursive Trees and Preferential Attachment Trees are Random Split Trees},
  journaltitle={Combinatorics, Probability and Computing},
  volume={28},
  number={1},
  pages={81--99},
  year={2019},
  doi={10.1017/S0963548318000226},
  eprint={1706.05487},
  eprinttype={arxiv},
  eprintclass={math.PR}
}

@article{gerstenberg2020words,
  author={Gerstenberg, Julian},
  title={General erased-word processes: Product-type filtrations, ergodic laws and {Martin} boundaries},
  journaltitle={Stochastic Processes and their Applications},
  volume={130},
  number={6},
  pages={3540--3573},
  year={2020},
  doi={10.1016/j.spa.2019.10.003}
}

@article{dong2006coagulation,
  author={Dong, Rui and Goldschmidt, Christina and Martin, James B.},
  title={Coagulation--fragmentation duality, {Poisson--Dirichlet} distributions and random recursive trees},
  journaltitle={The Annals of Applied Probability},
  volume={16},
  number={4},
  pages={1733--1750},
  year={2006},
  doi={10.1214/105051606000000655}
}

@article{geldbach2026continuum,
  author={Geldbach, David},
  title={Continuum asymptotics for tree growth models with uniform backward dynamics},
  journaltitle={ALEA. Latin American Journal of Probability and Mathematical Statistics},
  volume={23},
  pages={475--516},
  year={2026},
  doi={10.30757/ALEA.v23-18},
  eprint={2309.04336},
  eprinttype={arxiv},
  eprintclass={math.PR}
}

@inproceedings{kerov1994asymptotics,
  author={Kerov, Sergei},
  title={Asymptotics of large random {Young} diagrams},
  booktitle={Abstracts of the 6th Conference on Formal Power Series and Algebraic Combinatorics},
  publisher={DIMACS},
  year={1994},
  pages={285--294},
  url={https://fpsac-archive.github.io/FPSAC94/ARTICLES/34Kerov.pdf}
}

@article{fomin1994duality,
  author={Fomin, Sergey},
  title={Duality of graded graphs},
  journaltitle={Journal of Algebraic Combinatorics},
  volume={3},
  number={4},
  pages={357--404},
  year={1994},
  doi={10.1023/A:1022412010826}
}

@article{goodman2000martin,
  author={Goodman, Frederick M. and Kerov, Sergei V.},
  title={The {Martin} boundary of the {Young--Fibonacci} lattice},
  journaltitle={Journal of Algebraic Combinatorics},
  volume={11},
  number={1},
  pages={17--48},
  year={2000},
  doi={10.1023/A:1008739619211}
}

@article{bochkov2026ergodicity,
  author={Bochkov, Ivan and Evtushevsky, Vsevolod},
  title={Ergodicity of the {Martin} boundary of the {Young--Fibonacci} graph},
  journaltitle={Functional Analysis and Its Applications},
  volume={60},
  number={3},
  pages={209--227},
  year={2026},
  doi={10.1134/S1234567826030018}
}

@book{kerov2003asymptotic,
  author={Kerov, Sergei V.},
  title={Asymptotic Representation Theory of the Symmetric Group and its Applications in Analysis},
  series={Translations of Mathematical Monographs},
  number={219},
  publisher={American Mathematical Society},
  location={Providence, RI},
  year={2003},
  doi={10.1090/mmono/219}
}

@book{borodin2017representations,
  author={Borodin, Alexei and Olshanski, Grigori},
  title={Representations of the Infinite Symmetric Group},
  series={Cambridge Studies in Advanced Mathematics},
  number={160},
  publisher={Cambridge University Press},
  location={Cambridge},
  year={2017},
  doi={10.1017/CBO9781316798577}
}

@article{evans2017radix,
  author={Evans, Steven N. and Wakolbinger, Anton},
  title={Radix sort trees in the large},
  journaltitle={Electronic Communications in Probability},
  volume={22},
  eid={68},
  pages={1--13},
  year={2017},
  doi={10.1214/17-ECP77}
}

@online{zhang2026height,
  author={Zhang, Shengjun},
  title={Asymptotic height of {Plancherel} random trees},
  year={2026},
  eprint={2604.25877v1},
  eprinttype={arxiv},
  eprintclass={math.PR}
}

@online{zhang2026kingman,
  author={Zhang, Shengjun},
  title={Positivity Rigidity for {Grossman--Larson} Characters and the {Kingman} Face of the {Hoffman} Rooted-Tree Graph},
  year={2026},
  eprint={2609.31026v1},
  eprinttype={arxiv},
  eprintclass={math.CO}
}
\end{document}